\documentclass[12pt,a4]{article}
\usepackage[centertags]{amsmath}
\usepackage{amssymb,epsfig}
\usepackage{amsfonts}
\usepackage{amsthm}
\usepackage{mathrsfs}
\usepackage{newlfont}
\usepackage{makeidx}
\usepackage{latexsym}
\usepackage{graphicx}
\usepackage{pdfsync}

\theoremstyle{plain}
\newtheorem{theorem}{Theorem}
\newtheorem{corollary}{Corollary}
\newtheorem{lemma}{Lemma}
\newtheorem{proposition}{Proposition}
\theoremstyle{definition}
\newtheorem{definition}{Definition}
\newtheorem{example}{Example}

\newtheorem{remark}{Remark}
\newcommand{\smfrac}[2]{{\textstyle \frac{#1}{#2}}}

\usepackage[pdftex]{hyperref}

\numberwithin{equation}{section}

\begin{document}
\title{Quasistatic nonlinear viscoelasticity and gradient flows}
\author{J. M. Ball\footnote{Oxford Centre for Nonlinear PDE, Mathematical Institute, University of Oxford, Andrew Wiles Building, Radcliffe Observatory Quarter, Woodstock Road, Oxford
OX2 6GG,
 U.K.} and Y. \c{S}eng\"{u}l\footnote{Ozyegin University, Department of Natural and Mathematical Sciences, Ni\c{s}antepe Mah. Orman Sok. No: 13, 34794, Alemda\u{g}, Istanbul, Turkey}}
\date{}
\maketitle

\begin{abstract}
We consider the equation of motion for one-dimensional nonlinear viscoelasticity of strain-rate type under the assumption that the stored-energy function is 
$\lambda$-convex, which allows for solid phase transformations. We formulate this problem as a gradient flow, leading to existence and uniqueness of solutions. By approximating general initial data by those in which the deformation gradient takes only finitely many values, we show that under suitable hypotheses on the stored-energy function the deformation gradient is instantaneously bounded and bounded away from zero. Finally, we discuss the open problem of showing that every solution converges to an equilibrium state as time $t \to \infty$ and prove convergence to equilibrium under a nondegeneracy condition. We show that this condition is satisfied in particular for any real analytic cubic-like stress-strain function.

\end{abstract}

\tableofcontents

\section{Introduction}\label{intro}

In this paper we study the special case of the  equation for quasistatic one-dimensional nonlinear viscoelasticity of strain-rate type given by
\begin{equation}\label{e:qseqn}
\big(\sigma(y_{x})+\,y_{xt}\big)_{x}\,=\,0,\;\;x\in(0,1),
\end{equation}
with initial condition 
\begin{equation}\label{bc0}
y(x,0)=y_0(x).
\end{equation}
Here, $y(x,t)$ is the deformed position at time $t$ of a material point having position $x$ in the reference configuration,  $ \sigma(\cdot) = W'(\cdot)$ is the elastic part of the stress and 
\begin{equation}
\label{sef}
W(p)=\int_1^p\sigma(z)\,dz
\end{equation}
 is the stored-energy function.   In the case when the boundary conditions are
\begin{equation}\label{bc1}
y(0,t)\,=\,0,\,\,\,\, y(1,t)\,=\,\mu > 0,
\end{equation}
equation \eqref{e:qseqn} becomes, on setting $p(x,t)=y_x(x,t)$,
\begin{equation}\label{bc2}
p_{t} (x, t) = - \sigma(p(x,t)) + \displaystyle\int_{0}^{1} \sigma(p(y,t))\,dy,
\end{equation}
which is to be solved for initial condition
\begin{equation}\label{bc3}
p(x,0)=p_0(x),
\end{equation}where $p_0(x)=y_{0x}(x)$, so that $\int_0^1p_0(x)\,dx=\mu$. Our aim is to prove existence and uniqueness of the solutions for (\ref{e:qseqn}), \eqref{bc1} and show that these solutions converge to equilibrium states as time $t$ tends to infinity. Although \eqref{bc2} is a family of identical ordinary differential equations coupled together by a single average, it is by no means simple to analyze.

Starting from the 
general equation of motion for one-dimensional viscoelasticity of rate type given, for constant density $\rho>0$, by
\begin{equation}\label{e:visco}
\rho y_{tt}= \big(\sigma(y_{x})\,+\,S(y_{x},y_{xt})\big)_x, 
\end{equation}
\eqref{e:qseqn} is obtained
by setting $\rho=0$, making the choice   $\;S(y_x,y_{xt}) =\gamma y_{xt}$ for the viscoelastic part of the stress, where $\gamma>0$, and  scaling $t$ so that $\gamma=1$.

Equation \eqref{e:visco} is a prototype for the study of the dynamics of microstructure observed during solid phase transformations (see \cite{Sengul-thesis} for an extensive explanation). The main modelling assumption is that $\sigma$ is not a monotonic increasing function, so that 
$W$ 
is not convex. This can be thought of as the simplest model of a viscoelastic  solid, and it has been studied in many papers, for example  \cite{Daf}, \cite{Kut-Hicks}, {\cite{Ant-Seid},}  where both existence and uniqueness were obtained (see also  \cite{Po-Fe-81}, \cite{Po-Fe-82},  \cite{Demo-2000}, \cite{Tvedt} for treatments in three space dimensions). The study by \cite{Ericksen} of the corresponding equilibrium problem, which showed   that a non-monotone $\sigma$ could lead to reasonable predictions for one-dimensional models of solid phase transformations, motivated much of this work.

  The special case
\begin{equation}\label{jb1}
\rho y_{tt}=(\sigma(y_x)+y_{xt})_x
\end{equation}
corresponding to \eqref{e:qseqn} but  including inertia (i.e. $\rho>0$), was considered by \cite{And},   \cite{And-Ball}, \cite{Pego} and others (see \cite{Frie-Dolz} and \cite{Rybka92} for three-dimensional versions).   \cite{And} obtained an existence theory for weak solutions under   assumptions allowing for a non-monotone $\sigma$, based on a maximum principle for $y_x$.  \cite{And-Ball} then studied the asymptotic behaviour of the solutions as time $t$ goes to infinity, obtaining convergence to equilibrium, but only in the sense of Young measures, for both mixed boundary conditions 
\begin{equation}\label{jb2}
y(0,t)=0,\;\; (\sigma(y_x)+y_{xt})(1,t)=P,
\end{equation}
corresponding to the end $x=0$ of the bar being  fixed and the end $x=1$ being subjected to a force $P$, and displacement boundary conditions \eqref{bc1}, in the latter more difficult case under a nondegeneracy condition on $\sigma$ (see Section \ref{s:asymbehav}). Motivated by the maximum principle of Andrews, \cite{Pego} reformulated \eqref{jb1} as a semilinear parabolic partial differential equation coupled to an ordinary differential equation, and in this way proved convergence to equilibrium in the energy norm   under the boundary conditions \eqref{jb2}. The case of convergence to equilibrium for the boundary conditions \eqref{bc1} remains open in general, though as we remark in Section \ref{full} this in fact follows from the method of Pego under the nondegeneracy condition. 

The somewhat simpler equation \eqref{e:qseqn}, while presenting the same essential difficulties as the equation with inertia \eqref{jb1}, permits a somewhat simpler analysis, as well as stronger results. Under the assumption that $W$ is $\lambda-$convex, that is   $W(p)+\frac{1}{2}\lambda p^2$ is convex for some $\lambda>0$, we can apply the theory of $\lambda-$convex gradient flows of \cite{Brezis} (as used in  \cite{Mie-Ste}, \cite{Ros-Sav}, \cite{AGS})   to prove existence and uniqueness for \eqref{bc2}, \eqref{bc3} (see Section \ref{s:egft}). However it proves convenient to use a slightly different method based on the fact that 
\eqref{e:qseqn} has  solutions  taking only finitely many values that are described by a corresponding finite number of ordinary differential equations. Passing to the limit in these equations, using the same estimates as Brezis, enables one not only to prove existence and uniqueness for initial data $p_0\in L^2(0,1)$, and to justify a natural approach to computing solutions, but also to prove   universal bounds on the solutions independent of the initial data. This procedure is carried out in Section \ref{s:both-ends-fixed}. 

Using these bounds we are able, exploiting Helly's theorem to get relative compactness of positive orbits as in \cite{Serre}, to establish convergence to equilibrium for the case of displacement boundary conditions under a weakened nondegeneracy condition. In particular we prove convergence to equilibrium for a real analytic cubic-like $\sigma$. The general analysis of the nondegeneracy condition for real analytic $\sigma$ having a finite number of critical points raises interesting questions of algebraic geometry and complex analysis that will be addressed in a future paper.  Whether convergence to equilibrium holds in general, without any nondegeneracy condition, seems to be a very difficult problem. 
 Pego \cite{Pego-92} proves that convergence to equilibrium holds for   solutions taking finitely many values, using the theorem of  \cite{Hale-Massatt}, but for general solutions  the motion of phase boundaries presents extra difficulties. We note that a  variational scheme for a three-dimensional version of \eqref{e:qseqn} is discussed in \cite{Sengul-thesis}.

The  problem of convergence to equilibrium for (\ref{e:qseqn}) is   similar to that for the nonlinear diffusion equation studied in \cite{Pego-91}, and that of  homogeneous oscillations for a van der Waals fluid considered by \cite{Serre}. In both papers a form of the nondegeneracy  condition or some other additional hypothesis is used. We may also consider the $n$-dimensional form of \eqref{bc2}
\begin{equation}\label{bc4}
u_{t} =  - \sigma(u) + \frac{1}{| \Omega |} \int_{\Omega} \sigma(u) dx,
\end{equation}
where $\Omega \subset \mathbb{R}^{n}$ is a bounded domain with $n$-dimensional Lebesgue measure $|\Omega|$. This is the special case  
  $\varepsilon = 0$ of the equation 
\[u_{t} = \varepsilon \Delta u - \sigma(u) + \frac{1}{| \Omega |} \int_{\Omega} \sigma(u) dx,\]
studied in \cite{Rubin-Stern}  (see also  \cite{Ward-96}) as a model for phase separation, with, for example, $\sigma(u) = u^{3} - u$. They remark that this model can be obtained in the limit  $\alpha \to 0$ from the modification of the Cahn-Hilliard equation
\[\alpha u_{t} = \Delta ( \sigma(u) - \varepsilon \Delta u + \nu u_{t})\]
 proposed by   \cite{Novick-Cohen}, with the natural boundary conditions on $\partial \Omega$ given by
$$n \cdot \nabla (\sigma(u) - \varepsilon \Delta u - \gamma u_{t}) = n \cdot \nabla u = 0$$
and the mass constraint $$\int_{\Omega} u(x,t) dx = M.$$
In fact in \eqref{bc4} we may without loss of generality take $n=1$ and $\Omega=(0,1)$, so that our results for \eqref{bc2} are also valid for \eqref{bc4}. This is because of the result   that a separable and non-atomic measure space of measure one is isomorphic to the unit interval (\cite{Halmos-vonNeumann}, see also \cite{Rudolph}, \cite{Aaronson}). So there is a one-to-one measure preserving map $\varphi: (0,1) \to \Omega$, where $\Omega$ is endowed with   $n$-dimensional Lebesgue measure normalized so that $\Omega$ has measure one.  So the solution of \eqref{bc4} with initial data $u(x,0)=u_0(x)$ is given by $u(x,t)=p(\varphi^{-1}(x),t)$, where $p$ is the solution of \eqref{bc2}, \eqref{bc3} with initial data $p_0(x)=u_0(\varphi(x))$.


To avoid interpenetration of matter we require solutions of \eqref{e:qseqn} to satisfy $y_{x} \in (0,\infty)$. We handle this by assuming that $\sigma(p)\rightarrow -\infty$ as $p\rightarrow 0+$ and that $y_x(x,0)>0$ a.e. in $(0,1)$.  
We consider two sets of boundary conditions for \eqref{e:qseqn}, mixed and displacement.   For   mixed boundary conditions \eqref{jb2}
we assume without loss of generality that $P=0$, since the case $P\neq 0$ can be treated by replacing $\sigma$ by $\sigma-P$. We analyze this easy case in Section \ref{s:one-end-free}. The analysis helps motivate that carried out in Section   \ref{s:both-ends-fixed} for the more difficult set of displacement boundary conditions \eqref{bc1}.

\section{Mixed boundary conditions}\label{s:one-end-freea}

In this section we consider equation (\ref{e:qseqn}) with the boundary conditions  
\begin{equation}\label{mixed}
y(0,t)=0,\;\;(\sigma(y_x)+y_{xt})(1,t)=0.
\end{equation}
Thus we have to solve 
\begin{equation*}
\sigma(y_{x}(x,t))+\,y_{xt}(x,t)=0,\,\,\, \,\,\,\,x \in (0,1),
\end{equation*}
with $y(0,t)=0$.
Rewriting this equation in terms of $p:= y_{x}$, with $p_0(x)=y_x(x,0)$, the problem   becomes
\begin{displaymath}\label{s:one-end-free}
(\overline{P})\hspace{.5in}  \begin{array}{ll}
p_{t}(x,t)\,=\,-\sigma(p(x,t))\,\,\,\,\text{for}\,\,\,\,x \in (0,1),  \\ \\
p(x,0)\,=\,p_{0}(x),\\
\end{array}  
\end{displaymath}
from which $y$ can be recovered from 
$$y(x,t)=\int_0^xp(s,t)\,ds.$$
By a solution to the initial value problem $(\overline{P})$ on $[0,T], T>0$,we mean a function $p(x,t) \in C([0,T]; L^{1}(0,1)),$ which, for almost every $x \in (0,1),$ is such that $p(x,t)>0$ for $t>0$, $\sigma(p(x,\cdot))\in L^1(0,T)$, and satisfies the equality
\begin{equation}\label{e:soln-oef}
p(x,t)\,=\,p_{0}(x)\,-\,\int_{0}^{t} \sigma(p(x,\tau))\,d\tau
\end{equation}
for all $t \in [0,T]$. We have the following result for $(\overline{P}).$

\begin{theorem}\label{t:oef}
Assume that 
\begin{eqnarray*}
& \mathrm{(i)} & \sigma \colon (0,\infty) \rightarrow \mathbb{R}\,\,\,\text{is locally Lipschitz continuous, i.e. for any}\;\; C>1\;\; \text{there exists a}\\
&&\text{ positive constant}\;\; L=L(C)\;\; \text{such that}\;\;|\sigma(p)-\sigma(q)| \leq L(C)\,|p-q|\;\;\text{whenever}\\
&& \frac{1}{C} \leq |p|, |q| \leq C.
   \\
& \mathrm{(ii)} &  \sigma(p) > 0 \,\,\,\text{for sufficiently large}\,\,p.\\
& \mathrm{(iii)}&\sigma(p) \rightarrow -\infty\,\,\text{as}\,\,p \rightarrow 0{+}.
\end{eqnarray*}
Then, given any $p_{0} \in L^{1}(0,1), p_{0} \geq 0$ a.e. $x \in (0,1),$ there exists a unique solution $p$  to problem $(\overline{P})$ in $C([0,\infty); L^{1}(0,1)).$ 

Moreover there exists a continuous, nondecreasing function $P_{1}(t) > 0,$ independent of $p_{0},$ such that $p(x,t) \geq P_{1}(t)$ for a.e. $x \in (0,1)$ and for all $t > 0.$ 
If  further
\begin{equation}\label{e:iii}
\int_{p_{+} +1}^{\infty} \frac{d z}{\sigma(z)}\,<\,\infty\,,\,\,\text{where}\,\,\,p_{+}\,\,\,\text{is the largest root of}\,\,\,\sigma,
\end{equation}
then there exists a continuous, nonincreasing function $P_{2}(t) < \infty,$ independent of $p_{0}$, such that $p(x,t) \leq P_{2}(t)$ for a.e. $x \in (0,1)$ and for all $t > 0,$ and we have 
\[p \in C((0,\infty); L^{\infty}(0,1)).\]

As $t \to \infty,$ 
\[p(x,t) \to \bar{p}(x)\,\,\,\,\text{for a.e.}\,\,x \in (0,1),\]
where $\bar{p} \in L^{\infty}(0,1)$ and $\sigma(\bar{p}(x)) = 0$ for almost every $x \in (0,1).$

Moreover
\begin{equation}\label{e:conven}
\underset{t\rightarrow 0{+}}{\lim}
\,\int_{0}^{1}\,W(p(x,t))\,d x\,=\,\int_{0}^{1}\,W(p_{0}(x))\,dx. 
\end{equation}
\end{theorem}
\begin{proof} 
By fixing $x,$ we can reduce  our problem to consideration of  the ordinary differential equation 
\begin{equation}\label{ode}
\dot{p} = - \sigma(p)
\end{equation} 
for $p>0$. We have that  $p_0(x)\geq 0$   a.e. $x\in(0,1)$. If $p_0(x)>0$, since $\sigma$ is locally Lipschitz, there exists a unique local solution $p(x,t) \in C([0,T])$ of \eqref{ode} with $p(x,0)=p_0(x)$. 
Consider now the interval $[\varepsilon, C] \subset (0, \infty)$ where $\varepsilon > 0$ is sufficiently small and $C < \infty$ is sufficiently large. By assumptions $\mathrm{(ii)}, \mathrm{(iii)}$ it is clear that the direction field associated with $(\overline{P})$ points in the positive direction at $\varepsilon$ and in the negative direction at $C$. This shows that for \eqref{ode} with initial data $p_0$, 
\begin{equation}\label{e:invint}
p_{0} \in [\varepsilon, C]\,\,\,\,\Rightarrow\,\,\,\,p(t) \in [\varepsilon, C]\,\,\,\,\text{for all}\,\,\,t > 0.
\end{equation}
Therefore, $p(x,t)$ is a global solution of \eqref{ode}. If $p_0(x)=0$ we can define $p(x,t)$ as the unique solution of 
$$ \int_0^{p(x,t)}\frac{dz}{\sigma(z)}=-t,$$
for which $0<p(x,t)<C$ if $t>0$ and $\sigma(p(x,\cdot))\in L^1(0,T)$.
Since, for fixed $t\geq 0$, $p(x,t)$ is for any $\varepsilon>0$ a continuous function of $p_0(x)$ on the set $\{x\in(0,1):p_0(x)\geq \varepsilon\}$, it follows that $p(\cdot,t)$ is measurable. Note that for a.e. $x\in(0,1)$ and all $t\geq 0$ we have
\begin{equation}\label{bound1}
|p(x,t)|=p(x,t)\leq\max\{C,p_0(x)\}.
\end{equation}
Hence $p(\cdot,t)\in L^1(0,1)$ for all $t\geq 0$. Let $t_j\rightarrow t$ in $[0,\infty)$. Then by \eqref{bound1} and dominated convergence we have $p(\cdot,t_j)\rightarrow p(\cdot,t)$ in $L^1(0,1)$. 
Thus we have a unique solution $p(x,t)$ to problem $(\overline{P})$.  

To show the existence of the universal upper and lower bounds, it suffices to do this for the ordinary differential equation \eqref{ode}. Let $p_{-}$ and $p_{+}$ be the smallest and the largest roots of $\sigma$ respectively. 
For the lower bound,  define for $0<p<p_-$
\begin{equation*}
g(p) := \int_{0}^{p} \frac{- d z}{\sigma(z)}.
\end{equation*}
Then, by assumptions $\mathrm{(i)}$ and $\mathrm{(iii)}$, we have that $g$ is continuous, strictly monotonic increasing on $(0,p_{-})$ and $g(p) \to 0$ as $p \to 0{+}$. Let
\begin{equation}\label{e:P1}
P_{1}(t) := \min \Big\{ p_{-}, g^{-1}(t) \Big\}.
\end{equation}
If $p_{-} \leq p_{0},$ then $p_{-} \leq p(t)$ for all $t \geq 0$ so that $P_{1}(t) \leq p(t).$ If, on the other hand,  $0 \leq p_{0} < p_{-},$ then $\sigma(p(t))<0$ for all $t>0$, since roots of $\sigma$ are rest points. Hence
\[t=\int_{p_{0}}^{p(t)} \frac{- d z}{\sigma(z)}\,<\,\int_{0}^{p(t)} \frac{- d z}{\sigma(z)}\,=\,g(p).\]
Therefore  $p(t) \geq g^{-1}(t)$, giving $p(t) \geq P_{1}(t)$ in this case too.

For the upper bound define for $p>p_+$
\begin{equation*}
h(p):= \int_{p}^{\infty} \frac{d z}{\sigma(z)}.
\end{equation*}
Then, by \eqref{e:iii}, $h(p)$ is well defined and by assumptions $\mathrm{(i)}$ and $\mathrm{(ii)},$ it is continuous, strictly monotonic decreasing on $(p_{+}, \infty)$ and $h(p) \to 0$ as $p \to \infty$. Let
\begin{equation}\label{e:P2}
P_{2}(t) := \max \Big\{p_{+}+1, h^{-1}(t) \Big\}.
\end{equation}
Note that if
 $p_{0} \leq p_{+}+1$  then  $p(t) \leq p_{+}+1$  for all $t  \geq 0$. Hence  $p(t) \leq P_{2}(t)$ if $p_{0} \leq p_{+}+1.$ If, on the other hand, $p_{0} > p_{+}+1$, then $\sigma(p(t))>0$ for all $t>0$.  Hence
\[ 
\int_{p(t)}^{p_{0}}\,\frac{d z}{\sigma(z)}\,=\,t,\]
and so
\begin{equation*}
  t< \int_{p(t)}^{p_{0}}\,\frac{d z}{\sigma(z)}+
  \int_{p_{0}}^{\infty}\,\frac{d z}{\sigma(z)}=h(p(t)),
\end{equation*}
so that  $p(t) \leq h^{-1}(t).$ Thus $p(t) \leq P_{2}(t)$ also in this case. 

Thus if \eqref{e:iii} holds then 
\begin{equation}\label{e:universalbdd}
p(x,t) \in [P_{1}(t), P_{2}(t)]\quad \text{for}\quad t > 0,
\end{equation}
where $P_{1}(t)$ and $P_{2}(t)$ are given by (\ref{e:P1}) and (\ref{e:P2}) respectively.  Therefore $p(\cdot,t) \in L^{\infty}(0,1)$ for $t>0$. Moreover on any interval $[T_{1},T_{2}] \subset (0,\infty),$ $\sigma(p(x,t))$ is bounded, which implies that $p_{t}(x,t)$ is bounded and that there exists a constant $C$ such that
\[|p(x,t) - p(x,s) | \leq C\,|t - s|\,\,\,\text{for all}\,\,\,t,s \in [T_{1},T_{2}].\]
 Hence $p \colon (0,\infty) \to L^{\infty}(0,1)$ is continuous.

For a.e. $x\in(0,1)$ we have that $p(x,t)\rightarrow \bar p(x)$ for some root $\bar p(x)$ of $\sigma$. Since the roots of $\sigma$ are bounded, and since  $\bar{p}(\cdot)$ is measurable as it is the almost everywhere limit of a sequence of measurable functions, we have that $\bar p\in L^\infty(0,1)$. By \eqref{bound1} and dominated convergence we have $p(\cdot,t)\rightarrow \bar p$ in $L^1(0,1)$.  

Finally, note that for a.e. $x\in(0,1)$ we have 
\[\frac{\partial}{\partial t}W(p(x,t))\,=\,-\big(\sigma(p(x,t))\big)^{2}\,\leq\,0.\]
Therefore  (\ref{e:conven}) follows by monotone convergence   (whether or not $\int_0^1W(p_0(x))\,dx<\infty$). 
\end{proof}

\section{Displacement boundary conditions}\label{s:both-ends-fixed}

In this section we consider \eqref{e:qseqn} with displacement boundary conditions \eqref{bc1}, which as we have seen in the introduction is equivalent to   the problem 
\begin{eqnarray}
  p_{t}(x,t) &=& - \sigma(p(x,t))\,+\, \int_{0}^{1}\,\sigma(p(y,t))\,dy,  \label{peqn} \\
 (P)\hspace{.5in} p(x,0)&=&p_{0}(x) \geq 0 \:\;\;\text{a.e.} \,\,x \in (0,1), {}\nonumber  \\
 &&\hspace{-.6in}\int_{0}^{1}\,p(x,t)\,dx = \mu > 0. \nonumber {} 
\end{eqnarray}

\begin{definition}\label{d:solnP}
We say that $p=p(x,t)$ is a solution of the initial boundary-value problem $(P)$ on $(0,1) \times [0,T]$ if:
\begin{eqnarray*}
&\mathrm{(i)}& p \in C([0,T];L^{2}(0,1)),\,\,\text{with}\,\,\,p(\cdot,0) = p_{0}(\cdot),\\
&\mathrm{(ii)}& p(x,t) > 0\,\,\,\text{for a.e.}\,\,(x,t) \in (0,1) \times [0, T],\, \sigma(p(\cdot,t)) \in L^{1}(0,1)\,\,\,\text{for a.e.}\,\,t \in [0,T]\,\,\,\text{and}\\
& &  \sigma(p(x,t)) - \int_{0}^{1} \sigma(p(y,t))\,dy \in L^{1}(\tau, T)\,\,\,\text{for a.e.}\,\,x \in (0,1)\,\,\,\text{and all}\,\,\,\tau > 0,  \\
&\mathrm{(iii)}& p(x,t) = p(x,s) - \displaystyle\int_{s}^{t} \Big(\sigma(p(x,\tau)) - \int_{0}^{1} \sigma(p(y,\tau)) dy \Big) d\tau\,\,\, \text{for a.e.}\,\,x \in (0,1), \mbox{ for all }\,\\ & &s,t > 0.
\end{eqnarray*}
\end{definition}

\subsection{Assumptions}
We make the following general assumptions on the elastic stress.\vspace{.05in}
 
 \noindent (H1) $\sigma$ is locally Lipschitz continuous.\\
 (H2) $\sigma(p) \to -\infty$ as $p \to 0{+}$.\\
(H3) $W(p)$ is convex for $0\leq p<\theta$, for some $\theta>0$.\vspace{.05in}

\subsection{Finite-dimensional initial data}\label{s:fid}

In this section we study problem $(P)$ when the initial data is positive and takes finitely many values. That is, we have 
\begin{equation}\label{e:fid}
p_{0}(x)=\sum_{i=1}^{N}\,p_{0i}\,\chi_{E_{i}}(x),\quad p_{0i} > 0,
\end{equation}
where $\{E_{i}\}_{i = 1}^{N}$ is a partition of $(0,1)$ into disjoint measurable sets $E_{i}$ with $\mathrm{meas}(E_{i}) =\lambda_{i} > 0$ and $\sum_{i}\,\lambda_{i}=1$. The corresponding solution depends on the partition chosen, in particular on $N.$ We  denote this dependence by writing $p_{N}(x,t).$

\begin{theorem}\label{t:exisfid}
Assume {\rm (H1)-(H3)} hold and that the initial data $p_{0}$   is of the form $(\ref{e:fid})$. Then there exists a unique global solution $p_{N}(x,t)$ to $(P)$ given by
\begin{equation}\label{e:solnfid}
p_{N}(x,t)=\sum_{i=1}^{N}\,p_{i}(t)\,\chi_{E_{i}}(x),\,\,\,\,\,\,p_{i}(0)=p_{0i} > 0.
\end{equation}
\end{theorem}
\begin{proof}
Substituting $p_{N}(x,t)$ into problem $(P)$ gives
\begin{subequations}\label{e:fidode}
\begin{align}
& \dot{p_{i}}(t)=-\sigma(p_{i}(t))\,+\,\sum_{j=1}^{N}\,\lambda_{j}\,\sigma(p_{j}(t)),\,\,\,\,1 \leq i \leq N \label{e:fidodes} \\
& p_{i}(0) = p_{0i} > 0 ,\,\,\,\,\,\,\sum_{j=1}^{N} \lambda_{j} p_{j}(t) = \mu. \label{e:fidodeibc}
\end{align} 
\end{subequations}
Note that (\ref{e:fidodes}) is a finite system of ordinary differential equations with locally Lipschitz right-hand sides for $p_{i} >0.$ Hence, by the Picard-Lindel\"of Theorem (see, for example, \cite{Hartman}), it possesses a unique solution $p_{i}(t) \in C([0,T]), 1\leq i\leq N,$ for  sufficiently small $T$. This proves that $p_{N}(x,t)$ is a well-defined, unique local solution to problem $(P)$ with initial data satisfying (\ref{e:fid}). 

On the maximal interval of existence $[0, t_{\max})$ we have $\sum \lambda_{j} p_{j}(t) = \sum \lambda_{j} p_{j}(0)$ and so each $p_{j}(t)$ is uniformly bounded. Furthermore, if $p_{m}(0) < p_{n}(0)$ then $p_{m}(t) < p_{n}(t)$ for $0 \leq t < t_{\max},$ since otherwise there would be some $s$ with $p_{m}(s) = p_{n}(s)$ and we can solve the equation
\[\dot{q} = - \sigma(q) + c(t),\]
with $c(t)=\sum_{j=1}^{N}\,\lambda_{j}\,\sigma(p_{j}(t))$, 
backwards in time to get a contradiction. Assume $t_{\max} < \infty$. Then by standard properties of ordinary differential equations $\min_j p_j(t)\rightarrow 0$ as $t\rightarrow t_{\max}$. Let $i,k$ be such that  $p_{0i}\leq p_{0j}\leq p_{0k}$ for all $j$. Then, by the above ordering property, 
$$p_i(t)\leq p_j(t)\leq p_k(t) \;\;\mbox{for all }j\mbox{ and }t\in [0,t_{\max}).$$
Therefore $p_i(t)\rightarrow 0$ as $t\rightarrow t_{\max}$. Since $\mu>0$ we may assume without loss of generality that $\theta<\mu$. So $p_k(t)\geq\sum_{j=1}^N\lambda_jp_j(t)=\mu> \theta$. Since $p_j(t)$ is uniformly bounded  there is a constant $K\in\mathbb R$ such that  $\sigma(p_j(t))> K$ whenever $p_j(t)\geq\theta$. For $t$ sufficiently close to $t_{\max}$ we  have that $K\geq \sigma(p_i(t))$. For such $t$ either $p_j(t)< \theta$, in which case by (H3) we have that  $\sigma(p_j(t))-\sigma(p_i(t))\geq 0$, or $p_j(t)\geq \theta$, when  $\sigma(p_j(t))-\sigma(p_i(t))> K-\sigma(p_i(t))\geq 0$. Hence, since $p_k(t)>\theta$, 
$$\dot p_i(t)=\sum_{j=1}^N\lambda_j(\sigma(p_j(t))-\sigma(p_i(t)))>0,$$ contradicting $p_i(t)\rightarrow 0$. 
\end{proof}

\subsubsection{The Lower Bound}\label{s:lb}

In this subsection we prove that, independently of the initial data and $N$, $p_{N}(x,t)$ is instantaneously  bounded away from zero. We make the following additional assumptions, the first of which strengthens (H3):  \vspace{.05in}

\noindent (L1) There exists a constant $\alpha$ such that $\sigma'(p) \geq \alpha > 0 $ for $0\leq p < \theta$, for some $\theta>0$.\\
(L2) There exists a constant $c$ such that $\displaystyle \frac{\sigma(p)}{p} \geq c > 0$ for $p > 1/\theta$.
 \vspace{.05in}

\begin{proposition}\label{p:lb}
Assume that {\rm (H1), (H2)},  {\rm (L1)} and {\rm(L2)} hold.
Then there exist positive constants $C$ and $\varepsilon_{0}$  such that
\[\frac{\sigma(p)-\sigma(\delta)}{p-\delta}\,>\,C\,\,\,\,\,\text{for all}\,\,\,\,0< p \neq \delta,\,\,0<\delta \leq \varepsilon_{0}.\]
\end{proposition}
\begin{proof}
If the assertion was false, then it would in particular be false for $C = \varepsilon_{0} = \frac{1}{j}$ for all $j,$ and there would exist sequences $p_{j} \neq \delta_j \leq \frac{1}{j}$ such that
\begin{equation}\label{e:contrlb}
\frac{\sigma(p_{j})-\sigma(\delta_j)}{p_{j}-\delta_j}\,\leq\,\frac{1}{j}\,\,\,\,\,\text{for all}\,\,\,\,j.
\end{equation}
We can suppose that $p_{j}\,\rightarrow\,p_{\infty}\,\in\,[0,\infty]$ as $j\rightarrow \infty.$ Then we need to check three cases separately.\\
\indent (i)\,\,$p_{\infty}=0:$\,\,\,In this case by (L1) we have
\[\frac{\sigma(p_{j})-\sigma(\delta_j)}{p_{j}-\delta_j} \geq \alpha > 0\,\,\,\text{for all}\,\,\,j,\]
contradicting (\ref{e:contrlb}).\\
\indent (ii)\,\,$0 < \,p_{\infty}\,< \infty:$\,\,\,In this case $\sigma(p_{j})$ also stays finite by assumption (H1). Therefore, by (H2) we get
\[\frac{\sigma(p_{j})-\sigma(\delta_j)}{p_{j}-\delta_j}\rightarrow\,\infty\,\,\,\text{as}\,\,\,j \to \infty,\]
contradicting (\ref{e:contrlb}) again.\\
\indent (iii)\,\,$p_{\infty} =\,\infty:$\,\,\,In this case by (H2) and (L2) we immediately obtain
\[\frac{\sigma(p_{j})-\sigma(\delta_j)}{p_{j}-\delta_j}\geq \frac{\sigma(p_{j})}{p_{j}\left(1- \displaystyle\frac{\delta_j}{p_{j}}\right)} \geq c > 0\,\,\,\text{as}\,\,\,j \rightarrow \infty,\] which contradicts (\ref{e:contrlb}). 
\end{proof}

We now prove existence of a global lower bound.

\begin{theorem}\label{t:lowerbound}
Assume {\rm (H1), (H2), (L1)} and {\rm (L2)} hold. Then, there exists a continuous, nondecreasing $\epsilon(t)$, independent of $N$, with $\epsilon(0) = 0$ and $0 < \epsilon(t) < \mu$ for $t > 0,$ such that for any solution $p_N(x,t)$ to problem $(P)$ of the form \eqref{e:solnfid} we have $p_{i}(t) > \epsilon(t)$ for all $i$ and all $t > 0$.
\end{theorem}
\begin{proof}
Choose $\epsilon_{0}$ sufficiently small so that Proposition \ref{p:lb} holds, and such that $\epsilon_{0} < \theta, \epsilon_{0} < \mu$ and $\sigma(p) > \sigma(\epsilon_{0})$ for all $p > \epsilon_{0}$ (the latter is possible by (H2) and (L2)). Let $t_{0} = \frac{1}{C} \log\left(\frac{\mu}{\mu - \epsilon_{0}}\right) > 0,$ where $C$ is as in Proposition \ref{p:lb}. Define
\begin{displaymath}
\epsilon(t) = 
\left\{\begin{array}{ll}
\mu(1 - \exp(- C t)), \quad  & t \leq t_{0}, \\
\\
\epsilon_{0}, \quad & t > t_{0}.
\end{array}\right.
\end{displaymath}
We know that $p_{i}(0) > 0 = \epsilon(0)$ for all $i$. Suppose that the result is false. Let $\bar{t} > 0$ be the least value of $t$ with $\min_ip_{i}(t) = \epsilon(t).$ Suppose that $p_{i_{1}}(\bar{t}) = \cdots = p_{i_{M}}(\bar{t}) = \epsilon(\bar{t})$ and that $p_j(\bar t)>\varepsilon(\bar t)$ for $j\neq i_r, 1\leq r\leq M$. We have $M < N$ since $\sum \lambda_{j} p_{j}(\bar{t}) = \mu$ and $\epsilon_{0} < \mu.$ Then we have
\[\dot{p}_{i_{1}}(\bar{t}) = \sum_{j = 1}^{N} \lambda_{j} \big(\sigma(p_{j}(\bar{t})) - \sigma(p_{i_{1}}(\bar{t}))\big).\]
Case 1: Assume $\bar{t} \leq t_{0}$ so that $\epsilon(\bar{t}) \leq \epsilon_{0}.$Then,
\begin{eqnarray*}
\dot{p}_{i_{1}}(\bar{t})\,&=&\,\sum_{\underset{j \neq i_{r}}{j = 1}}^{N}\,\lambda_{j}\,\frac{(\sigma(p_{j}(\bar{t}))-\sigma(p_{i_{1}}(\bar{t})))}{p_{j}(\bar{t}) - p_{i_{1}}(\bar{t})}(p_{j}(\bar{t}) - p_{i_{1}}(\bar{t})) \\
 & > & C \,\sum_{j = 1}^{N}\,\lambda_{j}\,(p_{j}(\bar{t})-\epsilon(\bar{t})) \\
 & = & C ( \mu - \epsilon(\bar{t})) =  C \mu \exp(- C t).
 \end{eqnarray*}
 But ${p}_{i_{1}}(t) > \epsilon(t)$ for $0 < t < \bar{t}.$ So $\dot{p}_{i_{1}}(\bar{t}) \leq \dot{\epsilon}(\bar{t}) = \mu C \exp(- C t)$ giving a contradiction. \\
 Case 2: Assume $\bar{t} > t_{0}.$ Then $\epsilon(\bar{t}) = \epsilon_{0}$ and 
 \[\dot{p}_{i_{1}}(\bar{t}) = \sum_{j = 1}^{N} \lambda_{j} \big(\sigma(p_{j}(\bar{t})) - \sigma(\epsilon_{0}))\big).\]
 For $j = i_{1}, \ldots, i_{M}$ we have $p_{j}(\bar{t}) = \epsilon_{0},$ and for $j \neq i_{r}$ we have $p_{j}(\bar{t}) > \epsilon_{0}$ giving $\sigma(p_{j}(\bar{t})) > \sigma (\epsilon_{0})$. Hence $\dot{p}_{i_{1}}(\bar{t}) > 0$. However $p_{i_{1}}(t) \geq \epsilon_{0}$ for $t_{0} \leq t \leq \bar{t}$, which implies $\dot{p}_{i_{1}}(\bar{t}) \leq 0$, giving a contradiction.
\end{proof}

\noindent \emph{ {Weaker Lower Bounds}} \\

It is worth noting that one can obtain weaker bounds under somewhat weaker hypotheses.

\begin{proposition}
Assume that {\rm (H1), (H2), (L1)} hold and that $\sigma$ is bounded from below for large $p$. Then, for $m > 0$ sufficiently small, $p_{0j} \geq m$ for all $j$ implies $p_{j}(t) \geq m$ for all $j$.
\end{proposition}
\begin{proof}
Observe first that for any $r$ sufficiently small we have
\begin{equation}
\label{mon}(\sigma(r) - \sigma(q))(r - q) > 0\quad \text{for all}\;\;q \neq r.\end{equation}
This is because otherwise there would exist sequences $r_{j} \to 0$ and $q_{j} \neq r_{j}$ with 
\[(\sigma(r_{j}) - \sigma(q_{j}))(r_{j} - q_{j}) \leq 0\quad \text{for all}\;\;j.\]
 Therefore if $q_{j} \to 0,$ by (L1)  we get a contradiction. Note that   by (H2) $\sigma(r_{j}) \to - \infty$. If $q_{j}$ is bounded and bounded away from zero, then $r_{j} - q_{j}$ stays negative and so does $\sigma(r_{j}) - \sigma(q_{j}),$ and we get a contradiction. If $q_{j} \to \infty$, then by the assumption that $\sigma$ is bounded from below we again get a contradiction. 

Now suppose that $m>0$ is sufficiently small, that $p_j(0)\geq m$ for all $j$, but that there exist $i$ and $\bar t>0$ such that $p_i(\bar t)<m$. Since there must be such a $\bar t$ at which $\min_j p_j$ is strictly decreasing, we may assume that $p_j(\bar t)\geq p_i(\bar t)$ for all $j$ and that $\do p_i(\bar t)\leq 0$.    But
\[\dot{p}_{i}(\bar{t}) =  \sum_{j\in S(\bar t)} \frac{\lambda_{j}}{p_{j}(\bar{t}) - p_i(\bar t)} (\sigma(p_{j}(\bar{t})) - \sigma(p_i(\bar t)))(p_{j}(\bar{t})-p_i(\bar t))),\]
where $S(\bar t)=\{j:p_j(\bar t)> p_i(\bar t)\}$, which is not empty since $\sum_j\lambda_jp_j(\bar t)=\mu$ and $m$ is small. Hence  by \eqref{mon} with $q=p_i(\bar t), r=p_j(\bar t), r\neq q$ we obtain $\dot{p}_{i}(\bar{t}) > 0,$ giving a contradiction.
\end{proof}

\subsubsection{The Upper Bound}\label{s:ub}

In this section we show that, independently of the initial data and $N$, $p_{N}(x,t)$ is instantaneously bounded and stays bounded for all times. We assume that\vspace{.05in}
 
\noindent (U1) $W(p)$ is strictly convex for $p$ sufficiently large.\\
(U2) $\sigma(p) > 0$ for  $p$ sufficiently large, and $\displaystyle\int_{p_{+}+1}^{\infty} \frac{dz}{\sigma(z)} < \infty$, where $p_{+}$ is the largest root of $\sigma$ (which is in fact the assumption (\ref{e:iii}) in Theorem \ref{t:oef}).
 \vspace{.05in}

\begin{lemma}\label{superlemma}
Assume {\rm (H2), (U1)} and {\rm (U2)} hold. Then, 
\begin{equation}\label{e:superlinear}
\frac{\sigma(p)}{p}\,\rightarrow\,\infty\,\,\,\text{as}\,\,\,p \to \infty
\end{equation}
$($so that also {\rm (L2)} holds$)$.
\end{lemma}
\begin{proof}
First note the existence of $p_{+}$ follows from (H2) and the first part of (U2). By (U1) we have that $\sigma(p) \geq \sigma(z)$ if $p \geq z$ and $z$ is sufficiently large. Hence
\[\int_{\frac{p}{2}}^{p} \frac{1}{\sigma(z)}\,dz \geq \frac{p}{2\,\sigma(p)},\]
and the left-hand side tends to $0$ as $p \to \infty,$ proving the claim.
\end{proof}

\begin{lemma}\label{l:condforE}
Assume {\rm (H1), (H2), (U1)} and {\rm (U2)} hold. Then for all sufficiently large $\gamma > 0$ and $0 < p \leq \gamma$ we have
\begin{equation*}
\frac{\sigma(\gamma)}{2\,\mu} > \frac{\sigma(p)}{p} - \frac{\sigma(\gamma)}{\gamma}.
\end{equation*}
\end{lemma}
\begin{proof}
If the assertion was false, then there would exist sequences $\gamma_{j} \to \infty$ and $p_{j}$ with $0 < p_{j} \leq \gamma_{j}$ satisfying 
\begin{equation}\label{e:contrsuplin}
\frac{\sigma(\gamma_{j})}{2\,\mu} \leq \frac{\sigma(p_{j})}{p_{j}} - \frac{\sigma(\gamma_{j})}{\gamma_{j}}\,\,\,\text{for all}\,\,j.
\end{equation}
 We may assume that $p_j$ converges, possibly to $+\infty$, and so need to look at the following cases:\\
\indent $\mathrm{(i)}\,\, p_{j} \to 0:$\,\,In this case, (H2) and (\ref{e:superlinear})  immediately imply that the right-hand side of (\ref{e:contrsuplin}) goes to $- \infty$ as $j \to \infty.$ On the other hand, by (U2), the left-hand side is positive, giving a contradiction. \\
\indent $\mathrm{(ii)}\,\, p_{j} \to k > 0 :$\,\,In this case, by (H1), we know that $\frac{\sigma(p_{j})}{p_{j}}$ stays bounded as $j \to \infty$ and hence, by (\ref{e:superlinear}), the right-hand side goes to $- \infty.$ Again, by (U2), the left-hand side is positive, giving a contradiction.\\
\indent $\mathrm{(iii)}\,\, p_{j} \to \infty$ : By (U1) and the fact that $p_{j} <  \gamma_{j}$ we obtain
\[ \frac{\sigma(p_{j})}{2 \mu} \leq \frac{\sigma(\gamma_{j})}{2 \mu} \leq \frac{\sigma(p_{j})}{p_{j}} - \frac{\sigma(\gamma_{j})}{\gamma_{j}} < \frac{\sigma(p_{j})}{p_{j}},\]
giving a contradiction.
\end{proof}

We now prove the main result of this subsection which is the existence of a uniform upper bound. 

\begin{theorem}\label{t:upperbound}
Assume {\rm (H1)-(H3), (U1)} and {\rm (U2)} hold.  
Then there exists a continuous, nonincreasing function $E(t)$ for $t>0$, independent of $N$,  with $\lim_{t\rightarrow 0}E(t) = \infty$ and $E(t) > \mu$ for all $t > 0,$ such that for any solution $p_N(x,t)$ to problem $(P)$ of the form \eqref{e:solnfid} we have 
$p_{i}(t) < E(t)$ for all $i$ and all $t > 0.$
\end{theorem}
\begin{proof}
For $M > 0$ sufficiently large, we define
\begin{equation}\label{e:g(E)}
g(E) := \int_{M}^{E} \frac{d z}{\frac{\sigma(z)}{2 z} (z- 2\,\mu)} .
\end{equation}
Then $g$ is strictly increasing on $[M, \infty)$ and $g(\infty) := t_{0} < \infty.$ Let 
\begin{displaymath}
E(t) = 
\left\{\begin{array}{ll}
g^{-1}(t_{0} - t), \quad  & t \leq t_{0}, \\
\\
M, \quad & t > t_{0}.
\end{array}\right.
\end{displaymath}
Suppose that the claim is false. Then, there exists a least $\bar{t} > 0$ with $p_{i}(\bar{t}) = E(\bar{t})$ for some $i$.  Note that $E(t)$ is continuous and nondecreasing, and that $E(t) \geq M$ for all $t > 0 $.
From (\ref{e:fidode}) we obtain
\begin{eqnarray*}
\dot{p}_{i}(\bar{t}) & = & - \sigma(E(\bar{t}))\,+\,\sum_{j = 1}^{N}\,\lambda_{j}\,\sigma(p_{j}(\bar{t})).
\end{eqnarray*}
If $\bar{t} \leq t_{0},$ then by Lemma \ref{l:condforE} we have
\[\frac{\sigma(p_j(\bar{t}))}{p_{j}(\bar{t})} < \frac{\sigma(E(\bar{t}))}{2 \mu} + \frac{\sigma(E(\bar{t}))}{E(\bar{t})}.\]
Therefore,
\begin{eqnarray*}\label{e:tocontE}
\dot{p}_{i}(\bar{t})\,& < & \,- \sigma(E(\bar{t}))\,+\,
\sum_{j=1}^{N} \lambda_{j} p_{j}(\bar{t}) \left(\frac{\sigma(E(\bar{t}))}{2 \mu} + \frac{\sigma(E(\bar{t}))}{E(\bar{t})}\right) \\
& = & - \frac{\sigma(E(\bar{t}))}{2} + \mu \frac{\sigma(E(\bar{t}))}{E(\bar{t})}\\
& = & - \frac{\sigma(E(\bar{t}))}{2 E(\bar{t})}\,(E(\bar{t}) - 2 \mu).
\end{eqnarray*}
Note that $g(E(t)) = t_{0} - t$ for $t \leq t_{0}$. Hence $g'(E(t)) \dot{E}(t) = -1$. That is,
\[\dot{E}(\bar{t}) = - \frac{\sigma(E(\bar{t}))}{2 E(\bar{t})} (E(\bar{t}) - 2 \mu).\]
However, $\dot{p}_{i}(\bar{t}) \geq \dot{E}(\bar{t}),$ giving a contradiction.

If $\bar{t} > t_{0}$ on the other hand, then $E(\bar{t}) = M$ and
\[\dot{p}_{i}(\bar{t}) = \sum_{j = 1}^{N} \lambda_{j} (\sigma(p_{j}(\bar{t})) - \sigma(E(\bar{t}))).\]
Since $M$ is sufficiently large, we have $\sigma(p_{j}(\bar{t})) \leq \sigma(E(\bar{t}))$ for all $j,$ with strict inequality for some $j.$ Hence $\dot{p}_{i}(\bar{t}) < 0.$ However $p_{i}(t) \leq M$ for $t_{0} < t < \bar{t}.$ Therefore $\dot{p}_{i}(\bar{t}) \geq 0,$ giving a contradiction.
\end{proof}

\subsection{General initial data}
We now consider solutions of problem $(P)$ for general nonnegative initial data $p_0\in L^2(0,1)$.
\subsubsection{$\lambda$-convexity}
\label{lambdaconvexity}

We are particularly interested in $\lambda$-convex functionals, which are quadratic perturbations of convex functionals. 
 
\begin{definition}\label{d:lambdaconvex}
Let $K$ be a convex subset of a normed linear space $V$ with norm $\|\cdot\|$. Then a function $\phi:K\rightarrow {\mathbb R}\cup \{+\infty\}$  is $\lambda$-convex if
\begin{equation}\label{e:lambdaconvex}
  v\,\mapsto\,\phi(v)+\frac{\lambda}{2}\|v\|^{2}\,\,\text{is convex for some}\,\,\lambda \geq 0.
\end{equation}
\end{definition}

\noindent We now show that some of our assumptions imply $\lambda$-convexity.
\begin{proposition}\label{p:l-conv}
Assume {\rm (H1), (H3)} and {\rm (U1)} hold. Then $W$ is $\lambda$-convex on $[0,\infty)$ for some real $\lambda\,>\,0.$
\end{proposition}
\begin{proof}
  $W$ is $\lambda$-convex if and only if $z \mapsto W'(z) + \lambda z$ is nondecreasing on $(0,\infty)$ for some $\lambda > 0.$ For any sufficiently small $\theta > 0$, we know by (U1) and (H3) that if $p \leq \theta$ or $p \geq 1/ \theta,$ then $W'(p)$ is nondecreasing. Hence $W'(p)+\lambda\,p$ is nondecreasing for such values of $p$ for any $\lambda > 0.$ If, on the other hand, $p, q \in (\theta, 1/ \theta)$ with $p > q,$ then by (H1) we obtain
\begin{equation*}
 W'(p)+Lp-W'(q)-Lq\,\geq\,0,
\end{equation*}
where $L=L(\theta)> 0$ is the Lipschitz constant for $\sigma$. Choosing $\lambda\,=\,L $ gives the result.
\end{proof}

 We will follow a similar method to that of   \cite{Brezis} for the analysis of the evolution equations associated with monotone operators. Before stating the main result, we prove the following technical lemma using $\lambda$-convexity.

\begin{lemma}\label{l:estimate}
Assume that $W$ is $\lambda$-convex with corresponding $\lambda \geq 0.$ Then, for any $p>0, q>0,$ we have
\[\big(\sigma(p)-\sigma(q)\big)\big(p-q\big)\,\geq \,- \lambda\,\big(p - q)^{2}.\]
\end{lemma}
\begin{proof}
Since $W$ is $\lambda$-convex for $\lambda \geq 0,$ we have $W'(p) + \lambda\,p$ is nondecreasing in $p.$ Therefore, without loss of generality  taking $p > q,$ we have
\[W'(p) + \lambda\,p \geq W'(q) + \lambda\,q  \quad \Leftrightarrow \quad (\sigma(p) - \sigma(q))\,(p-q) \geq - \lambda \,(p - q)^{2}\]
as claimed.
\end{proof}

\begin{proposition}\label{p:propagation}
Assume {\rm (H1)-(H3)} and {\rm (U1)} hold, and 
\begin{equation}\label{e:convinid}
p_{0 N}\,\rightarrow\, p_{0}\,\,\,\text{in}\,\,\,L^{2}(0,1)\quad \text{as}\quad N \to \infty,
\end{equation}
where $p_{0N}=\sum_{i=1}^N\lambda^N_i\chi_{E_i^N}$ is of the form \eqref{e:fid}, with corresponding solution $p_N=p_{N}(x,t)$ satisfying $(P)$. Then, there exists a $p=p(x,t)$ with $p(0)=p_0$  such that 
\begin{equation}\label{e:pisin}
p_{N}\,\rightarrow \,p\,\,\,\,\text{in}\,\,\,\,C([0,T];L^{2}(0,1))\,\,\,\,\text{as}\,\,\,\,N \to \infty.
\end{equation}
\end{proposition}
\begin{proof}\,\,
Take $p_{N}$ and $p_{M}$ satisfying $(P)$ with corresponding initial data $p_{0N}$ and $p_{0M}.$ Then from \eqref{peqn} we obtain
\begin{eqnarray*}
\big(p_{N}(x,t)-p_{M}(x,t)\big)_{t} &=& -\big(\sigma(p_{N}(x,t))-\sigma(p_{M}(x,t))\big)+ \\
& & \qquad  \quad \quad +\int_{0}^{1}\big(\sigma(p_{N}(y,t))-\sigma(p_{M}(y,t))\big)\,dy.
\end{eqnarray*}
By the boundary conditions, this implies
\begin{eqnarray*}
&   & \frac{1}{2}\frac{d}{dt}\int_{0}^{1}|\,p_{N}(x,t)-p_{M}(x,t)\,|^{2}\,dx = \nonumber \\
&   & \qquad \qquad = -\int_{0}^{1}(\sigma(p_{N})-\sigma(p_{M}))(p_{N}-p_{M})\,dx + \nonumber \\
&   & \qquad \qquad \qquad \qquad +\left(\int_{0}^{1}(p_{N}-p_{M})\,dx \right)\,\left(\int_{0}^{1}(\sigma(p_{N})-\sigma(p_{M}))\,dy\right) \nonumber \\
&  & \qquad \qquad =  -\int_{0}^{1}(\sigma(p_{N})-\sigma(p_{M}))(p_{N}-p_{M})\,dx. \label{e:forsigmabound}
\end{eqnarray*}
Hence by Proposition \ref{p:l-conv} and Lemma \ref{l:estimate} we obtain
\begin{equation*}
\frac{1}{2}\frac{d}{dt}\int_{0}^{1}|p_{N}(x,t)-p_{M}(x,t)|^{2}\, dx \leq \lambda \int_{0}^{1} | p_{N}(x,t) -p_{M}(x,t) |^{2} dx.
\end{equation*}
By Gr\"{o}nwall's inequality   this gives
\begin{equation}\label{e:gronwall}
\int_{0}^{1}|p_{N}(x,t)-p_{M}(x,t)|^{2}\,dx\,\leq\,\exp(2 \lambda t)\,\int_{0}^{1}|p_{0N}(x)-p_{0M}(x)|^{2}\,dx\,.
\end{equation}
By (\ref{e:convinid}) this shows that
$p_{N} $ is a Cauchy sequence in $C([0,T];L^{2}(0,1))$ and so converges to $p $ with $p(0) = p_{0}$ proving that \eqref{e:pisin} holds.
\end{proof}

\begin{theorem}\label{t:exisuniq}
Assume that {\rm (H1)} and {\rm (H2)} hold, $W$ is $\lambda$-convex and 
\begin{equation}\label{e:assposp}
p_{0} \in L^{2}(0,1),\,\,\, p_{0}(x)\,\geq\,0\,\,\,\text{for}\,\,\,\mathrm{a.e.}\,\,x \in (0,1),\,\,\,\int_0^1p_0(x)\,dx=\mu.
\end{equation}
Then there exists a unique solution $p=p(x,t)$ to problem $(P)$, and the map $(p_0,t)\mapsto p(\cdot,t)$ is continuous from $L^2(0,1)\times [0,\infty)$ to $L^2(0,1)$. Furthermore, $p$ satisfies the energy equation
\begin{equation}
\label{energyeq}
\int_0^1W(p(x,t))\,dx=\int_0^1W(p(x,\tau))\,dx-\int_\tau^t\int_0^1p_s^2(x,s)\,dx\,ds
\end{equation}
for any $t\geq\tau>0$.
\end{theorem}
\begin{proof}
(\textit{Existence}) \,\,
  Since $p_0$ is  nonnegative and measurable there exists a nondecreasing sequence $q_{0N}$ of nonnegative measurable functions, each taking only $N$ values, not necessarily distinct, each on sets of positive measure, converging to $p_0$ almost everywhere (cf. \cite{Bartle}). Thus $\int_0^1q_{0N}\,dx\rightarrow \mu$ and so 
$$ p_{0N}=\mu\left(\frac{q_{0N}+\frac{1}{N}}{\int_0^1q_{0N}\,dx+\frac{1}{N}}\right)$$
defines a sequence of strictly positive functions of the form \eqref{e:fid} satisfying $\int_0^1p_{0N}\,dx=\mu$ and $p_{0N}\rightarrow p_0$ in $L^2(0,1)$.
By Proposition \ref{p:propagation} we know that this implies the existence of a $p=p(x,t)$ such that \eqref{e:pisin} is satisfied.
It is therefore enough to show that $p(x,t)$ satisfies the conditions in Definition \ref{d:solnP}. 
From Proposition \ref{p:l-conv} we know that $W$ is $\lambda$-convex for some $\lambda > 0.$ Hence, we have
\[W(\mu) + \frac{\lambda}{2}\,\mu^{2} \geq W(p_{N}) + \frac{\lambda}{2}\,p_{N}^{2} + (\mu - p_{N}) (\sigma(p_{N}) + \lambda\,p_{N}).\]
Integrating with respect to $x$ gives
\[\int_{0}^{1} W(p_{N})\,dx \leq W(\mu) - \frac{\lambda}{2} \mu^{2} + \frac{\lambda}{2} \int_{0}^{1} p_{N}^{2} dx - \frac{d}{dt} \frac{1}{2} \| p_{N} - \mu \|_{2}^{2}.\]
Integrating with respect to time we obtain
\begin{eqnarray*}
\int_{0}^{T} \int_{0}^{1} W(p_{N}) dx\,dt & \leq & T \left(W(\mu) - \frac{\lambda}{2} \mu^{2} \right) + \frac{\lambda}{2} \int_{0}^{T} \int_{0}^{1} p_{N}^{2} dx\,dt  \\ 
& & \quad \quad - \frac{1}{2} \| p_{N}(T) - \mu \|_{2}^{2} + \frac{1}{2} \| p_{N}(0) - \mu \|_{2}^{2}.
\end{eqnarray*}
Since $p_{N}$ is bounded in $C([0,T] ; L^{2}(0,1)),$ for any finite $T$ we have that the right-hand side of the above inequality is bounded. Therefore,
\begin{equation}\label{e:W(pN)-bdd}
\int_{0}^{T} \int_{0}^{1} W(p_{N}) dx\,dt \leq M < \infty \quad (\text{independent of}\,\,N).
\end{equation}
On the other hand, denoting the inner product in $L^{2}(0,1)$ by $(\cdot, \cdot),$ for each $t >0$ we have
\begin{eqnarray*}
t\,\| \dot{p}_{N} \|_{2}^{2} & = & - \big(t\,\dot{p}_{N}, \sigma(p_{N}) - \int_{0}^{1} \sigma(p_{N}) dy \big) \\
& = & - t\,\frac{d}{dt} \int_{0}^{1} W(p_{N})\,dx.
\end{eqnarray*}
Integrating both sides with respect to time gives
\begin{eqnarray*}
\int_{0}^{T} t\,\| \dot{p}_{N} \|_{2}^{2} dt & = & - \int_{0}^{T} \left[ \frac{d}{dt} \left(t\,\int_{0}^{1} W(p_{N}) dx \right) - \int_{0}^{1} W(p_{N}) dx \right] dt \\
& = & \int_{0}^{T} \int_{0}^{1} W(p_{N}) dx\,dt - T\,\int_{0}^{1} W(p_{N}(x,T))\,dx.
\end{eqnarray*}
By \eqref{e:W(pN)-bdd} the first term is bounded. But $\lambda$-convexity implies that $W(p) + \lambda p^{2}$ is bounded from below for sufficiently large $\lambda > 0.$ Hence, the second integral is also bounded from below independently of $N$. As a result we obtain
\begin{equation}\label{e:t-bound}
\int_{0}^{T} t\,\int_{0}^{1} \Big( \sigma(p_{N}) - \int_{0}^{1} \sigma(p_{N}) dy \Big)^{2} dx\,dt \leq C(T) < \infty,
\end{equation}
where $C(T) > 0$ is a constant depending on $T.$ Let us define $Q := (\tau, T) \times (0, 1)$ where $\tau > 0.$ From \eqref{e:t-bound} we immediately have that for an appropriate subsequence, not relabelled, 
\begin{equation}\label{e:weak-L2-conv}
\sigma(p_{N}) - \int_{0}^{1} \sigma(p_{N}) dx \rightharpoonup \chi \,\,\, \,\text{in} \,\,\, \,L^{2}(Q),
\end{equation}
where$\int_0^1\chi(x,t)\,dx=0$ for a.e. $t\in(\tau,T)$. 
Suppose $v \in L^{2}(Q)$ is such that $\int_{0}^{1} v(x, t)\,dx = \mu$ and $\sigma(v(\cdot, t)) \in L^{1}(0,1)$ for a.e. $t \in (\tau, T),$ and $\sigma(v(x,t)) - \int_{0}^{1} \sigma(v(y,t)) dy \in L^{2}(Q)$. Then, by $\lambda$-convexity we have, for each $t\in(\tau,T)$,
\begin{eqnarray*}
- \lambda\,\| p_{N} - v \|_{2}^{2} & \leq & \Big(\sigma(p_{N}) - \sigma(v), p_{N} - v \Big) \\
& = & \Big(\sigma(p_{N}) - \int_{0}^{1} \sigma(p_{N}) dy - \sigma(v) + \int_{0}^{1} \sigma(v) dy, p_{N} - v \Big).
\end{eqnarray*}
Hence,
\[\int_{\tau}^{T} \Big(\sigma(p_{N}) - \int_{0}^{1} \sigma(p_{N})\,dy - \sigma(v) + \int_{0}^{1} \sigma(v)\,dy, p_{N} - v \Big) dt \geq - \lambda \int_{\tau}^{T} \| p_{N} - v \|_{2}^{2} \,dt.\] 
By \eqref{e:pisin} and \eqref{e:weak-L2-conv}, passing to the limit as $N \to \infty$ gives
\begin{equation}\label{e:chi-ineq}
\int_{\tau}^{T} \Big(\chi - \sigma(v) + \int_{0}^{1} \sigma(v) dy, p - v \Big) dt \geq - \lambda \int_{\tau}^{T} \| p - v \|_{2}^{2}\,dt.
\end{equation}
For a.e. $t \in (\tau, T)$ we have that $\chi(\cdot, t) \in L^{2}(0,1)$. Now, we choose $v(x,t)$ to minimize (at this time $t$) the functional
\begin{equation}\label{e:min-v}
I(v) = \int_{0}^{1} \Big(W(v) + \lambda\,v^{2} - ( 2\,\lambda\,p(x,t) + \chi(x,t))v \Big) \,dx
\end{equation}
in  $L^2(0,1)$ subject to $\int_{0}^{1} v(x,t) \,dx = \mu.$ The minimizer exists and is unique since the integrand is strictly convex in $v.$ We claim that the minimizer satisfies the Euler-Lagrange equation
\begin{equation}\label{e:EL}
\sigma(v) + 2\,\lambda\,v - 2\,\lambda\,p - \chi = c(t) \in \mathbb{R}
\end{equation}
and in particular that $v(x,t) > 0$ for a.e. $x \in (0,1).$

First of all, we claim that for each $t$ there is a unique solution $v(\cdot, t)$ to \eqref{e:EL} with $\int_{0}^{1} v(x,t) dx = \mu.$ To see this note that $h(v) = \sigma(v) +2 \lambda v$ satisfies $h(v) \to - \infty$ as $v \to 0{+}$, and that $h$ is strictly increasing with $ h(v) \to \infty$ as $v \to \infty.$ Given $c \in \mathbb{R}$ define 
\[\eta(c) = \int_{0}^{1} h^{-1}(c + 2 \lambda p + \chi)\, dx.\]
Since $h(v) \geq a + \lambda v$ for $v\geq 1$, where $a=\sigma(1)+\lambda$, it follows that  $v \geq h^{-1}(a +  \lambda  v)$  for $v\geq 1$, i.e. $0 \leq h^{-1}(s) \leq \max(1, \frac{s-a}{\lambda})$. Thus $\eta(c)$ is a well-defined real number for each $c$, and $\eta(c)$ is continuous and strictly increasing. By monotone convergence $\eta(c) \to 0$ as $c \to - \infty$ and $\eta(c) \to \infty$ as $c \to \infty.$ Hence there exists a unique $c(t)$ such that $\eta(c(t)) = \mu.$ Setting \[v(x,t) = h^{-1}(c(t) + 2 \lambda p(x,t) + \chi(x,t))\]
we get a unique solution to \eqref{e:EL} satisfying $\int_0^1v(x,t)\,dx=\mu$. If $u \in L^{2}(0,1)$ and $\int_{0}^{1} u \,dx = \mu,$ then
\[W(u) + \lambda u^{2} - (2 \lambda p + \chi) u \geq W(v) + \lambda v^{2} - (2 \lambda p + \chi) v + [\sigma(v) + 2 \lambda v - (2 \lambda p + \chi)] (u - v),\] 
with strict inequality if $u\neq v$, and integrating we get $I(u) \geq I(v).$ Hence $v$ is the unique minimizer. Note that $\eta$ is measurable in $t$, thus so is $c(t)$ and hence $v$ is measurable in $x$ and $t.$ Also testing with $u =\mu$ we have
\[\int_{0}^{1} \left( W(v) + \lambda v^{2} - (2 \lambda p + \chi) v\right)\, dx \leq \int_{0}^{1}\left( W(\mu) + \lambda \mu^{2} - (2 \lambda p + \chi) \mu\right)\,dx .\]
Since $W(v) + \lambda v^{2} \geq b + \frac{\lambda}{4} v^{2}$ for some $b$, this implies that
\[\int_{0}^{1} v^{2}(x,t)\, dx \leq \text{constant} + f(t)\]
where $f \in L^{1}(\tau, T)$ and hence $v \in L^{2}(Q).$ Also $\sigma(v(\cdot, t)) \in L^{2}(0,1)$ for a.e. $t \in (\tau, T)$, $c(t) = \int_{0}^{1} \sigma(v(x,t)) dx$ and $\sigma(v) - \int_{0}^{1} \sigma(v)\, dx \in L^{2}(Q).$ So, from \eqref{e:chi-ineq} and \eqref{e:EL} we get 
\[- \int_{\tau}^{T} 2 \lambda \| p - v \|_{2}^{2}\,dt \geq - \lambda \int_{\tau}^{T} \| p - v \|_{2}^{2} dt.\]
This implies
\begin{equation}\label{e:chi}
v = p\,\,\,\text{and}\,\,\,\chi = \sigma(p) - \int_{0}^{1} \sigma(p)\,dx.
\end{equation}
From the finite-dimensional problem, for any $\psi \in L^{2}(0,1)$, we have
\[\big(p_{N}(\cdot,t), \psi(\cdot) \big) = \big(p_{N}(\cdot, \tau), \psi(\cdot) \big) - \int_{\tau}^{t} \Big(\sigma(p_{N}) - \int_{0}^{1} \sigma(p_{N}) \,dy, \psi(\cdot) \Big) ds.\] 
Passing to the limit as $N \to \infty$ leads to
\[\big(p(\cdot, t), \psi(\cdot) \big) = \big(p(\cdot, \tau), \psi(\cdot) \big) - \int_{\tau}^{t} \Big(\sigma(p) - \int_{0}^{1} \sigma(p) \,dy, \psi (\cdot)\Big) ds\]
so that
\[p(\cdot, t) = p(\cdot, \tau) - \int_{\tau}^{t} \Big(\sigma(p(\cdot, s)) - \int_{0}^{1} \sigma(p(y, s))\,dy\Big) ds\,\,\,\,\text{in}\,\,\,\,L^{2}(0,1).\]
That is, 
\[p(\cdot, t) - \int_{\tau}^{t}\Big(\sigma(p(\cdot,s)) - \int_{0}^{1} \sigma(p(y,s)) dy \Big) ds\,\,\,\text{is independent of}\,\,\,t.\]
Hence for a.e. $x$ we have
\[p(x, t) = p(x, s) - \int_{s}^{t} \Big(\sigma(p(x, \tau)) - \int_{0}^{1} \sigma(p(y, \tau))\,dy\Big) d\tau\] 
for all $s,t>0$. Thus we have existence of a solution.\\

\noindent (\textit{Uniqueness and continuous dependence on initial data}) \,\,
If $p_{1}$ and $p_{2}$ are two solutions,
then for $i = 1,2,$ we get
\[\dot{p}_{i}(x,t)= - \sigma(p_{i}(x,\tau)) + \int_{0}^{1}\sigma(p_{i}(y,\tau))\, dy \,d\tau\quad \text{for a.e.}\,\,x \in (0,1).\]
After subtracting these two equalities and arguing as in the proof of Proposition \ref{p:propagation} we get an inequality similar to (\ref{e:gronwall}) which is
\[\|p_{1}(\cdot,t) - p_{2}(\cdot,t)\|^{2}_{2}\,\leq\,\exp(2 \lambda t)\,\|p_{1}(\cdot,0) - p_{2}(\cdot,0)\|^{2}_{2}.\]
This proves the asserted continuous dependence on the initial data, from which uniqueness follows immediately.\\

\noindent ({\it Energy equation})  By \cite[Lemma 3.3, p.73]{Brezis} (applied to $W + \lambda p^{2}$ and $p^{2}$) the energy equation \eqref{energyeq} holds for all $t, \tau > 0$. 
\end{proof}
\begin{corollary}
\label{semiflow}
Under the assumptions of Theorem \ref{t:exisuniq}, the solution $p(t)=T(t)p_0$ generates a semiflow $\{T(t)\}_{t\geq 0}$ on  the closed subset $$X=\{q\in L^2(0,1):q\geq 0 \;\mbox{ a.e. } x\in(0,1), \int_0^1q\,dx=\mu\}$$ of $L^2(0,1)$, i.e.    $T(t):X\rightarrow X$  for each $t\geq 0$ and satisfies {\rm(i)} $T(0)=\mbox{identity}$, {\rm (ii)} $T(s+t)=T(s)T(t)$ for all $s,t\geq 0$ and {\rm (iii)} $(p_0,t)\mapsto T(t)p_0$ is continuous from $X\times [0,\infty)$ to $X$.
\end{corollary}
\begin{proof}
We just need to check that $\int_0^1p(x,t)\,dx=\mu$ for all $t\geq 0$. This follows by integration of  (iii) and using (i), (ii)  in Definition \ref{d:solnP}.
\end{proof}

\begin{corollary}\label{c:bounds}
Assume {\rm (H1), (H2), (L1), (U1), (U2)} hold. Then for any $p_0\in L^2(0,1)$ with $p_0(x)\geq 0$ for a.e. $x\in (0,1)$ and $\int_0^1p_0(x)\,dx=\mu$ there exists a unique solution $p$ to Problem {\rm (P)}, and for all $t>0$ the universal lower and upper bounds 
$$\varepsilon(t)\leq p(x,t)\leq E(t) \;\;\mbox{a.e. }x\in (0,1),$$
 hold, where $\varepsilon(t)$ and $E(t)$ are as in Theorems \ref{t:lowerbound}, \ref{t:upperbound} respectively.
\end{corollary}
\begin{proof} Note that by Lemma \ref{superlemma} and Proposition \ref{p:l-conv} we have that (L3) holds and $W$ is $\lambda$-convex. 
By Theorems \ref{t:lowerbound} and \ref{t:upperbound} we have existence of  a lower bound $\epsilon(t)$ and an upper bound $E(t)$, respectively, which are both independent of $N.$ Hence, passing to the limit in $\epsilon(t) \leq p_{N}(x,t) \leq E(t)$ as $N \to \infty$ gives the claim for any $t > 0$.
\end{proof}
\begin{remark}
\rm In fact under the stronger hypotheses of the Corollary the proof of existence is much easier, since we can use the upper and lower bounds on $p_N$ to prove that $$\sigma(p_N)-\int_0^1\sigma(p_N)\,dx\;\;\rightarrow\;\;  \sigma(p) - \int_{0}^{1} \sigma(p)\,dx \;\;\;\;\mbox{   strongly in } L^2(Q).$$
\end{remark}
\begin{remark} The existence of the universal lower and upper bounds implies that it is impossible to solve problem $(P)$ backwards in time on any time interval if the initial data $p_0$ does not satisfy $\varepsilon\leq p_0(x)\leq \frac{1}{\varepsilon}$ a.e. $x\in(0,1)$ for some $\varepsilon>0$. This is not surprising in view of the derivation of \eqref{peqn} from \eqref{e:qseqn}.
\end{remark}
\begin{remark}\label{rmk2}We list various additional properties satisfied by the solution $p$ to Problem (P) whose existence was proved in Theorem \ref{t:exisuniq}.\label{s:ex-rem}\medskip

\noindent (a)\;\; $p \in W^{1,2}(\tau, T; L^{2}(0,1)).$ 
Indeed, we know that
\[\int_{0}^{1} W(p_{N}(t))\, dx + \int_{\tau}^{t} \| \dot{p}_{N}(s) \|^{2} ds = \int_{0}^{1} W(p_{N}(\tau))\, dx.\]
Hence $p_{N}$ is bounded in $W^{1,2}(\tau, T; L^{2}(0,1)),$ which implies that $\dot p=-\sigma(p)+\int_0^1\sigma(p)\,dy\in L^2((\tau,T)\times(0,1))$ and thus $p \in W^{1,2}(\tau, T; L^{2}(0,1))$ for any $0<\tau<T$. It follows that the derivative 
$\dot{p}(t)$ exists for a.e. $t>0$ (see \cite[p.145]{Brezis}).\medskip

\noindent (b)\;\;If $W(0) = \infty,$ then $p(x,t) > 0$ a.e. $x \in (0,1)$ and all $t > 0.$ 
Indeed, we have the estimate 
\[t \int_{0}^{1} W(p_{N})\, dx \leq \int_{0}^{t} \int_{0}^{1} W(p_{N}(x,s))\, dx\, ds \]
for $t\geq 0$, since $\int_0^1W(p_N)\,dx$ is nonincreasing in $t$. Therefore, using the fact that $W(v) \geq b - c v^{2}$ for constants $b$ and $c>0$, we have that 
\begin{eqnarray*} t \int_{0}^{1} W(p_{N})\, dx &\leq& \int_0^t\int_0^1 (W(p_N)+cp_N^2-b)\,dx\,ds-\int_0^t\int_0^1(cp_N^2-b)\,dx\,ds\\
&\leq&\int_0^T\int_0^1 (W(p_N)+cp_N^2-b)\,dx\,ds -\int_0^t\int_0^1(cp_N^2-b)\,dx\,ds\\
&\leq& M_1<\infty,
\end{eqnarray*}
for all $t\in [0,T]$, where $M_1$ is independent of $t\in[0,T]$ and we have used \eqref{e:W(pN)-bdd} and Proposition \ref{p:propagation}. Hence, using again the estimate $W(v) \geq b - c v^{2}$
and Fatou's Lemma, we get that 
\[t \int_{0}^{1} W(p(x,t))\, dx \leq M_1 < \infty\]
for all $t\in [0,T]$.
Thus, for any $t > 0,$ we have $\text{meas}\{x \colon p(x,t) = 0\} = 0$. Now, suppose that for $x$ in a set $E\subset(0,1)$ of positive measure, $p(x, t(x)) = 0$ for some $t(x) > 0$. Since $\text{meas}\,E>0$ there exists $x'\in E$ such that $\text{meas}\{x\in E:p_0(x)\leq p_0(x')\}>0$. By Lemma \ref{l:monotonicity} below $0\leq p(x, t(x'))\leq p(x',t(x'))=0$  for all $x$ with $p_{0}(x) \leq p_{0}(x')$,   contradicting $\text{meas}\,\{x:p(x,t(x'))=0\}=0$. Hence, for a.e. $x,$ $p(x,t) > 0$ for all $t > 0.$

\end{remark}

\subsection{Relation with the Theory of Gradient Flows}\label{s:egft}

In this section, we analyze problem $(P)$ using the well-developed existence theory of gradient flows.

\subsubsection{Classical Theory of Gradient Flows}

Let $H$ be a  Hilbert space with inner product $(\cdot,\cdot)$ and norm $\| \cdot \|.$ For given $T> 0$ and $f \colon (0,T) \rightarrow H,$ the gradient flow equation is given by
\begin{displaymath}\label{e:GF}
 \begin{array}{ll}
  \hspace{.6in} \dot{u}(t)\,+\,\partial \phi(u(t))\,\ni\,f(t)\,\,\,\text{for a.e.}\,\,t \in (0,T) & {} \\
 (GF)\\
  \hspace{.6in} u(0)\,=\,u_{0} & {}
\end{array}
\end{displaymath}
where $\phi:H \rightarrow (-\infty,\infty]$ is a proper and lower semicontinuous functional  with effective domain $D(\phi)\,=\,\{u \in H:\,\phi(u)\,<\,\infty \}$ and $\partial \phi \colon D(\partial \phi) \subset H \rightarrow 2^{H}$ is its Fr\'{e}chet subdifferential with   corresponding domain
\begin{equation}\label{e:subdif-dom}
D(\partial \phi) = \{u \in H : \partial \phi(u) \neq \emptyset \}.
\end{equation}
Let us recall that the functional $\phi$ is said to be \textit{proper} if $D(\phi)\,\neq\,\emptyset$ and the \textit{Fr\'{e}chet subdifferential} $\partial \phi$ of $\phi$ at a point $u \in D(\phi)$ is defined as
\begin{equation}\label{e:subdiff}
v \in \partial \phi(u)\,\,\,\,\Leftrightarrow \,\,\,\,\underset{\omega \rightarrow u}{\liminf}\,\, \frac{\phi(\omega)-\phi(u)-(v,\omega -u)}{\| \omega - u \|}\geq\,0.
\end{equation} 
When $\phi$ is assumed to be convex, $ \partial \phi$ is a maximal monotone operator (for the definition see e.g. \cite{Rockafellar}). In this case, existence, uniqueness and regularity of solutions for problem $(GF)$ follow from the well-known theory of nonlinear semigroups in Hilbert spaces developed by \cite{Brezis},   \cite{Cran-Pazy} and   \cite{Kom}. \begin{definition}\label{d:solGF}
Let $f \in L^{1}(0,T;H).$ A function $u \in C([0,T]; H)$ is called a solution of $(GF)$ if $u$ is differentiable a.e. on $(0,T),$ $u(t) \in D(\phi)$ for a.e. $t \in (0,T)$, and there exists $g(t) \in \partial \phi(u(t))$ for a.e. $t \in (0,T)$ such that
$\dot{u}(t) + g(t) = f(t)$ for a.e. $t \in (0,T).$
\end{definition}
When $\phi$ is $\lambda$-convex, the Fr\'{e}chet subdifferential can be characterized by
\begin{displaymath}
\begin{array}{lll}
   v \in \partial \phi(u) & \Leftrightarrow & {}
\end{array}
\left\{ \begin{array}{l}
   u \in D(\phi)\,\,\text{and}  \\ \phi(\omega)-\phi(u)-(v,\omega - u) \geq - \frac{\lambda}{2} \| \omega - u \|^{2}\,\,\,\textrm{for all}\,\,\omega \in H.
\end{array}\right.
\end{displaymath}
This case is covered by   \cite{Brezis} as a Lipschitz  perturbations of the convex case and $u_{0}$ is assumed to be only in $\overline{D(\partial \phi)}$ (see Remark \ref{r:subdif-dom}). His result \cite[Prop. 3.12]{Brezis} gives as a special case that:
\begin{quote}
Assume that $\phi$ is $\lambda$-convex and $u_{0} \in \overline{D(A)}.$ Then there exists a unique  solution to the equation
\[\frac{d u}{d t} + \partial\phi(u)   \ni 0, \quad u(0) = u_{0}.\]
\end{quote}

\subsubsection{Equivalence of the theories}

In this section we show that the existence theory we developed in Section \ref{s:both-ends-fixed} can also be obtained by the theory of gradient flows for $\lambda$-convex functionals, and vice versa. 

\begin{definition}\label{d:phi-domain}
 We define the functional $\phi$ on $L^{2}(0,1)$ as
\begin{displaymath}
\phi(p) = \left\{ \begin{array}{ll}
\displaystyle\int_{0}^{1}W(p)\,dx, & \quad \textrm{if}\,\,\,\displaystyle\int_{0}^{1}p\, dx=\mu, \,p \geq 0\,\,\text{a.e.}  \\
+\infty, & \quad \textrm{otherwise}
\end{array} \right.
\end{displaymath}
and its effective domain as
\begin{equation*}\label{e:domain}
D(\phi(p))\,=\,\left\{p \in L^{2}(0,1):\,p \geq 0\,\,\text{a.e.},\,\int_{0}^{1}W(p)\,dx\, < \infty,\,\int_{0}^{1}\,p\,dx =\mu \right\}.
\end{equation*}
\end{definition}
We now prove the fundamental result necessary for the proof of the equivalence of the theories.

\begin{proposition}\label{p:equivalenceoftheories}
 Assume $W$ is $\lambda$-convex and that {\rm (H1), (H2)} hold. Then, 
\begin{equation*}
\xi \in \partial \phi(p)\,\,\Leftrightarrow\,\,p \in D(\phi), p > 0 \,\,\text{a.e.}, \sigma(p) \in \displaystyle L^{2}(0,1), \xi =\sigma(p)-c\,\,\,\textrm{for a constant}\,\,c.
\end{equation*}
\end{proposition}
\begin{proof}
For sufficiency, let $v \in L^{2}(0,1)$ and consider $v$ with $\phi(v) < \infty.$ Then, by $\lambda$-convexity of $W$ we obtain
\begin{eqnarray*}
& & \phi(v) - \phi(p) - (\xi, v - p) + \frac{\lambda}{2} \| v - p \|^{2} \\
& & = \int_{0}^{1}\left( W(v) - W(p) - \xi (v - p) + \frac{\lambda}{2} (v - p)^{2}\right)  dx\\
& & = \int_{0}^{1} \left( W(v) - W(p) - (\sigma(p) - c) (v - p) + \frac{\lambda}{2}(v - p)^{2}\right)  dx \\
& & = \int_{0}^{1}\left(  W(v) + \frac{\lambda}{2} v^{2} - W(p) - \frac{\lambda}{2} p^{2} - (\sigma(p) + \lambda p)(v - p) \right) dx \geq 0.
\end{eqnarray*}
For necessity, let $\xi\in\partial\phi(p)$ and first suppose that $p = 0$ on a set $A$ of positive measure (so that $W(0) < \infty$). Then, $p \geq \mu$ on a set of positive measure and so there exists a set $B$ of positive measure on which $\mu \leq p \leq M < \infty$. We choose $v = p + \epsilon z$ where
\begin{displaymath}
z(x) = \left\{\begin{array}{lll}
1,&  x \in A \\ 
- \frac{\text{meas} A}{\text{meas} B}, & x \in B \\
0, & \text{otherwise}.
\end{array}\right.
\end{displaymath}
Then, $\displaystyle\int_{0}^{1} z(x) dx = 0$ and
\[\int_{0}^{1}\left( W(v) - W(p) - \xi (v - p) + \smfrac{\lambda}{2} (v - p)^{2}\right)  dx \geq 0.\]
This is equivalent to 
\begin{eqnarray*}
&&  \int_{A}\left( W(\varepsilon) - W(0) - \epsilon \xi + \smfrac{\lambda}{2} \epsilon^{2}\right)\,dx \\ 
&& \qquad + \int_{B}\left( W\left(p - \epsilon \smfrac{\text{meas} A}{\text{meas} B} \right) - W(p) + \epsilon \xi \left(\smfrac{\text{meas} A}{\text{meas} B}\right) + \smfrac{\lambda}{2} \epsilon^{2} \left(\smfrac{\text{meas} A}{\text{meas} B}\right)^{2}\right)\,dx \geq 0.\end{eqnarray*}
Dividing by $\epsilon$ and letting $\epsilon \to 0$ we get a contradiction by (H2). This shows that $p$ cannot equal  $0.$ So, let $p > 0$ and choose $v = p + \epsilon \,z$ where $z$ satisfies
\begin{displaymath}
\left\{\begin{array}{lll}
z \in L^{\infty}(0,1), \\ 
z=0\,\,\,\text{outside}\,\,\,\displaystyle E_{\tau}:=\left\{x \in (0,1):\,\tau<p(x)<\smfrac{1}{\tau},\,\tau > 0\right\},\,\\  \displaystyle\int_{E_{\tau}}z\,dx\,=\,0.
\end{array}\right.
\end{displaymath}
Then, clearly $v \in D(\phi)$ and $v \to p$ in $L^{2}(0,1)$ as $\epsilon \to 0$. Hence
\begin{eqnarray*}
 \int_{0}^{1} \big(W(p + \epsilon z) - W(p) - \xi \epsilon z + \smfrac{\lambda}{2}  \epsilon^{2} z^{2} \big) dx \geq 0 .
\end{eqnarray*}
Dividing by $\epsilon$ and letting $\epsilon \to 0$ gives 
$$\int_{E_{\tau}}\,(W'(p)-v)\,z\,dx\,\geq\,0.$$
Repeating the above calculation with $-z$ instead of $z,$ one gets the same inequality for $-z.$ Therefore, we must have
\begin{equation*}
 \int_{E_{\tau}}(\sigma(p)-\xi)\,z\,dx\,=\,0.
\end{equation*}
Therefore
\begin{equation*}
\sigma(p) - \xi = c(\tau)\,\,\,\,\text{for a.e.}\,\,x \in E_{\tau},
\end{equation*}
where $c(\tau)$ is a constant depending on $\tau.$ However, by definition, $E_{\tau}$ is an increasing set which implies that $c(\tau)$ cannot depend on $\tau.$ Therefore
\begin{equation*}
\sigma(p) - \xi = c
\end{equation*}
must hold for a constant $c$ as required.
\end{proof}

\begin{remark}\label{r:subdif-dom} 
By Proposition \ref{p:equivalenceoftheories} and \eqref{e:subdif-dom} we have that 
\begin{eqnarray*} 
& & D(\partial \phi) = \left\{ p \in L^{2}(0,1) \colon \sigma(p) \in L^{2}(0,1), p > 0\,\,\textrm{a.e.}, \right. \\
& & \qquad \qquad \qquad \qquad \qquad \qquad \qquad \left.\int_{0}^{1} W(p) dx < \infty, \int_{0}^{1} p\,dx = \mu \right\}
\end{eqnarray*}
and hence
\[\overline{D(\partial \phi)} = \left\{ p \in L^{2}(0,1) \colon p \geq 0\,\,\textrm{a.e.}, \int_{0}^{1} p \,dx = \mu \right\}.\]
Note that we did not assume the initial energy to be bounded in Section \ref{s:both-ends-fixed}. This is equivalent to assuming $p_{0} \in \overline{D(\partial \phi)}$ as in the above result of Brezis. 
\end{remark}

We can now establish the equivalence of  the existence theory  developed in Section \ref{s:both-ends-fixed} for one-dimensional nonlinear viscoelasticity with that of the theory of gradient flows.

\begin{theorem}
Assume that $W$ is $\lambda$-convex, and \rm (H1) and \rm (H2) are satisfied. Then, any solution $p(\cdot,t)$ of problem $(P)$ is a solution of $(GF)$ with $f(t) \equiv 0,$ and vice versa.
\end{theorem}
\begin{proof}
We take $f(t) \equiv 0$ in $(GF)$ and consider a solution $p(t)$ to $(GF).$ By Definition \ref{d:solGF}, we know that $p$ is differentiable a.e. on $(0,T)$, $p(t) \in D(\phi(p))$, and there exists a $g(t) \in \partial \phi(p(t))$ such that 
\[- \dot{p}(t) = g(t)\,\,\,\,\text{a.e.}\,\, \text{in}\,\, (0,T).\] 
By Proposition \ref{p:equivalenceoftheories}, we must have $p(x,t)=p(t)(x)> 0$ a.e., $\sigma(p(t)) \in L^{2}(0,1)$, and 
\[g(t) = \sigma(p(t))-c(t),\,\,\int_{0}^{1}p(t)\,dx=\,\mu.\]
Therefore,
\[\int_{0}^{1} \dot{p}(t) dx = - \int_{0}^{1} g(t) dx \quad \Rightarrow\quad c(t) = \int_{0}^{1}\sigma(p(t))\,dx\] so that we have
\[\dot{p}(t) = - \sigma(p(t)) + \int_{0}^{1} \sigma(p(t))\,dx.\]
According to Definition \ref{d:solnP} this shows that $p(t)$ is a solution of $(P)$.

Conversely, for all $s, t \in (0,T),$ any solution $p(t)$ of problem $(P)$ satisfies
\[p(t)- p(s) =  - \int_{s}^{t} \Big(\sigma(p(\tau)) - \int_{0}^{1} \sigma(p(\tau)) dx \Big) d\tau.\]
Since $p$ is differentiable for a.e. $t \in (0,T)$ (see Remark \ref{rmk2}) we can divide both sides by $(t - s)$ and let $s \to t.$
This gives
\[\dot{p}(t) = - \sigma(p(t)) + \int_{0}^{1} \sigma(p(t)) dx\,\,\,\,\,\text{for a.e.}\,\,\,t \in (0,T).\]
Since $p \in W^{1,2}(\tau, T; L^{2}(0,1))$ we have $- \sigma(p(t)) + \int_{0}^{1} \sigma(p(t)) dx \in L^{2}(0,1)$ for a.e. $t>0$.  However, the integral is a function of $t$ only, and hence we get $\sigma(p(t)) \in L^{2}(0,1).$ Now  we can set $c(t) = \int_{0}^{1}\,\sigma(p(t))\,dx$, so that $p(t)$ is a solution of $(GF)$ according to Definition \ref{d:solGF}. 
\end{proof}

\section{Asymptotic Behaviour of Solutions}\label{s:asymbehav}

In this section we investigate the asymptotic behaviour of solutions for system $(P)$.  For the rest of the paper we assume that (H1), (H2), (L1), (U1), (U2) hold. By Corollary \ref{semiflow} we know that solutions to $(P)$ generate a semiflow   $\{T(t)\}_{t\geq 0}$ on $$X=\{q\in L^2(0,1):q\geq 0 \;\mbox{ a.e. } x\in(0,1), \int_0^1q\,dx=\mu\}.$$

\subsection{Equilibrium solutions}

Any equilibrium solution $\bar{p} \in L^{2}(0,1)$ for problem $(P)$ satisfies
\begin{displaymath}
\left\{\begin{array}{ll}
  0\,=-\sigma(\bar{p}(x))\,+\,\displaystyle\int_{0}^{1}\,\sigma(\bar{p}(y))\,dy\, & {}\\
  \displaystyle\int_{0}^{1}\,\bar{p}(x)\,dx=\mu & {}
\end{array}\right.
\end{displaymath}
giving $\sigma(\bar{p}(x)) = \,$constant. Therefore, we can define the set of equilibrium solutions as
\begin{equation}\label{e:equilibrium}
\mathcal{E}_{\mu}:= \left\{\bar{p} \in L^{2}(0,1) :\bar p(x) >0 \;\mbox{ a.e. },\; \sigma(\bar{p}(x)) = C\,\,\,\text{for some}\,\,\,C\in \mathbb{R},\,\,\,\int_{0}^{1} \bar{p}(x)\,dx = \mu \right\}.\nonumber
\end{equation}

\begin{proposition}
  $\bar{p}(x) \equiv \mu$ is the unique equilibrium solution if and only if $$\mu\not\in(\min\sigma^{-1}(c),\max\sigma^{-1}(c))$$ for any $c\in\mathbb R$.
In particular, $\bar{p} \equiv \mu$ is the unique equilibrium if $\sigma$ is strictly monotone increasing. If $\bar{p}(x) \equiv \mu$ is not the unique equilibrium, then there are uncountably many equilibria.
\end{proposition}
\begin{proof}
If $\min \sigma^{-1}(c) < \mu < \max \sigma^{-1}(c)$ for some $c \in \mathbb{R}$, then 
\[\mu = s\,\min \sigma^{-1}(c) + (1 - s)\,\max \sigma^{-1} (c)\,\,\,\text{for}\,\,\,0 < s < 1,\]
and an uncountable family of equilibria is given by
\begin{displaymath}
p_{E}(x) =\left\{\begin{array}{ll}
  \max \sigma^{-1}(c), \quad &  x \in E^{c}, \\
  \min \sigma^{-1}(c), \quad & x \in E
\end{array}\right.
\end{displaymath}
for any measurable $E \subset (0,1)$ with $\text{meas}(E) = s.$

Conversely, if there is an equilibrium $\bar{p} \neq \mu$, then since $\sigma(\bar{p}) = c$ for some $c$ and $\int_{0}^{1} \bar{p} = \mu,$ there exist subsets $E_{1}, E_{2}$ of $(0,1)$ of positive measure such that $\bar{p}(x) < \mu$ on $E_{1}, \bar{p}(x) > \mu$ on $E_{2}$ so that $\min \sigma^{-1}(c) < \mu < \max \sigma^{-1}(c).$
\end{proof}

\begin{remark}
Note that there can still be uncountably many equilibria if $\sigma$ is monotone and constant on an interval containing $\mu$ as an interior point.
\end{remark}
 
\subsection{Convergence to the set of equilibria} 
 
\begin{proposition}\label{p:convtimeder}
 Let $p(x,t)$ be a solution of problem $(P).$ Then we have \[p_t(\cdot,t)\rightarrow 0 \;\mbox{ in } L^2(0,1) \;\mbox{ as } t\rightarrow\infty.\]
\end{proposition}
\begin{proof} We use a similar argument to \cite{And-Ball}.
By Theorem \ref{t:exisuniq} we have that for any $\tau>0$
\begin{equation}\label{finiteint}\int_\tau^\infty\|p_t(\cdot,t)\|_2^2\,dt<\infty.\end{equation}
If $\tau>0$ and $t\geq s\geq \tau$ we have that 
\begin{eqnarray*}
\left|\|p_t(\cdot,t)\|_2^2-\|p_t(\cdot,s)\|_2^2\right|&=&
\left|\int_0^1\left(\sigma(p(x,t))^2-\sigma(p(x,s))^2\right)\,dx \right.\\
&&\hspace{.5in}\left. -\left(\int_0^1\sigma(p(y,t))\,dy\right)^2+\left(\int_0^1\sigma(p(y,s))\,dy\right)^2\right|\\
&\leq & C \int_0^1 |p(x,t)-p(x,s)|\,dx
= C \int_0^1 \left| \int_s^t \frac{d}{d \tau} p(x, \tau)\,d\tau \right| dx \\
& \leq& \int_0^1 \int_s^t |p_{\tau}(x,\tau)|\, d\tau\,dx 
 \leq C |t -s |,
\end{eqnarray*}
where we used (H1) and the universal bounds in Corollary \ref{c:bounds}. Hence $\|p_t(\cdot,t)\|_2^2$ is uniformly continuous on $[\tau,\infty)$, so that by \eqref{finiteint} $p_t(\cdot,t)\rightarrow 0$ in $L^2(0,1)$ as $t\rightarrow\infty$.
\end{proof}

We can now prove relative compactness of positive orbits, by noting that the equation is invariant to rearrangement of the initial data (cf. \cite{Serre}).
\begin{lemma}\label{l:monotonicity}
Let $p_0\in X$ be such that $p_0(x)$ is a nondecreasing function of $x$ on $(0,1)$ (that is, $p_0$ has a nondecreasing representative). Then $p(x,t)=(T(t)p_0)(x)$ is a nondecreasing function of $x$ on $(0,1)$ for all $t\geq 0$.
\end{lemma}
\begin{proof}
The claim  holds in the finite-dimensional case (cf. proof of Theorem \ref{t:exisfid}). For the general case, note that  we can approximate $p_{0}$ in $X$ by nondecreasing $p_{0N}$. Then $p_N(\cdot,t)=T(t)p_{0N}$ is nondecreasing, and the result follows since $p_N\rightarrow p$ in $C([0,T]; L^2(0,1))$ for any $T>0$. 
\end{proof}
\begin{proposition}\label{p:conveqsubseq}
Given any $p_0\in X$, the $\omega$-limit set 
$$\omega(p_0)=\{\chi\in X:T(t_j)p_0\rightarrow \chi \;\mbox{ for some sequence }t_j\rightarrow \infty\}$$
is a nonempty, compact subset of the set $\mathcal E_\mu$ of equilibria, and 
\begin{equation}\label{attraction}
{\mathrm{dist}}\, (T(t)p_0,\omega(p_0))\rightarrow 0 \;\mbox{ as } t\rightarrow\infty.
\end{equation}
\end{proposition}
\begin{proof}
By   \cite{Ryff} (see also \cite{Chong}),   for any real integrable function $f$ defined on the interval $(0,1)$ there exists a measure-preserving map $\delta \colon [0,1] \to [0,1]$ such that $f$ can be written in terms of its nondecreasing rearrangement $f^{*}$ as 
\begin{equation*}
f = f^{*} \circ \delta\,\,\,\,\text{for a.e.}\,\,x \in (0,1).
\end{equation*}
Moreover, for any measure-preserving map $\delta \colon [0,1] \to [0,1]$ and integrable function $f \in L^{1}(0,1),$ we have
\begin{equation}\label{e:mpminv}
\int_{0}^{1} f(x)\,dx = \int_{0}^{1} f(\delta(y))\,dy.
\end{equation} 
We apply this to the initial data $p_0$, thus obtaining a nondecreasing rearrangement $p_0^*$ such that $p_0(x)=p^*_0(\delta(x))$ for some measure-preserving map  $\delta(x) \colon [0,1] \to [0,1]$. 
 By Lemma \ref{l:monotonicity} the solution  $p^*(\cdot,t)=T(t)p_0^*$ with initial data $p_0^*$ is nondecreasing for each $t\geq 0$.  We claim that $p(x,t)=p^*(\delta(x),t)$. Indeed, since $p^*$ is a solution of $(P)$, we have that 
\begin{equation}
\label{Pstar}
p^*(\zeta,t)=p^*(\zeta,\tau)-\int_\tau^t\left(\sigma(p^*(\zeta,s))-\int_0^1\sigma(p^*(\zeta',s))\,d\zeta'\right)\,ds
\end{equation}
 for a.e. $\zeta\in(0,1)$ and all $t\geq\tau>0$. Setting $\zeta=\delta(x)$ in \eqref{Pstar}, and observing that, since $\delta$ is measure-preserving,  $\int_0^1\sigma(p^*(\zeta',s))\,d\zeta'=\int_0^1\sigma(p^*(\delta(x),s))\,dx$, we see that $p^*(\delta(x),t)$ is a solution to $(P)$, and so by uniqueness $p(x,t)=p^*(\delta(x),t)$ as claimed.

 By the universal upper  bound we  know that $p^*(\cdot,t)$ is bounded in  $L^{\infty}(0,1)$ for sufficiently large $t$. Since $p^*(\cdot,t)$ is nondecreasing it follows that $p^*(\cdot,t)$ is a function of uniformly bounded variation. Applying Helly's Selection Theorem (see e.g. \cite[p. 222]{Natanson}) it follows that  for any sequence $t_j\rightarrow\infty$ there exists a subsequence $ t_{j_k} $ and a bounded nondecreasing $q \in L^{2}(0,1)$ such that
\[p^*(\cdot,t_{j_k})\,\rightarrow\,q(\cdot)\,\,\,\text{in}\,\,\,L^{2}(0,1)\,\,\,\text{as}\,\,\,k \to \infty.\]
Hence $p(\cdot,t_{j_k})\rightarrow q(\delta(\cdot))$ as $k\rightarrow \infty$. Thus the positive orbit $\gamma^+(p_0)=\bigcup_{t\geq 0}T(t)p_0$ is relatively compact in $X$, and thus by standard results (see e.g. \cite[p. 36]{Hale}) $\omega(p_0)$ is nonempty, compact  and \eqref{attraction} holds. 
By the universal lower and upper bounds and the continuity of 
  $\sigma$, if $\chi=\lim_{j\rightarrow\infty}T(t_j)p_0\in\omega(p_0)$ then
\[\sigma(T(t_{j})p_0)\,\rightarrow\,\sigma(\chi)\,\,\,\text{in}\,\,\,\,L^{2}(0,1)\,\,\,\,\text{as}\,\,\,j \to \infty.\]
By  Proposition \ref{p:convtimeder} we also have that 
\[\sigma(T(t_j)p_0)-\int_{0}^{1}\sigma(T(t_j)p_0)\,dy\,\rightarrow\,0\,\,\,\,\text{in}\,\,\,\,L^{2}(0,1)\,\,\,\,\text{as}\,\,\,\,j \to \infty ,\]
and thus $\sigma(\chi)=\int_0^1\sigma(\chi)\,dx$ and
  $\chi \in \mathcal{E}_{\mu}.$
\end{proof}
\begin{corollary}\label{monsigma}
If $\sigma$ is monotone (not necessarily strictly) and $p_0\in X$ then $T(t)p_0\to\bar p$ as $t\to\infty$ for some equilibrium $\bar p\in\mathcal E_\mu$.
\end{corollary}
\begin{proof}
If $\sigma$ is monotone then $\{T(t)\}_{t\geq 0}$ is a contraction semigroup on $X$, since
$$\frac{d}{dt}\frac{1}{2}\|p(t)-q(t)\|_2^2=-(\sigma(p)-\sigma(q),p-q)\leq 0$$
for any two solutions $p,q$. By Proposition \ref{p:conveqsubseq}, $\omega(p_0)$
consists only of equilibria. Let $\bar p\in\omega(p_0)$. Then $\lim_{t\to\infty}\|T(t)p_0-\bar p\|_2=l$ for some constant $l$. But since $\bar p\in \omega(p_0)$ we have $l=0$ and $T(t)p_0\to \bar p$ as $t\to\infty$, as required.
\end{proof}

By a {\it global attractor} $A$ of a semiflow $\{T(t)\}_{t\geq 0}$ on a metric space $X$ is meant a compact, invariant (i.e. $T(t)A=A$ for all $t\geq 0$) set that attracts bounded sets.

\begin{theorem}
There exists a  global attractor in $X$ for the   semiflow $\{T(t)\}_{t\geq 0}$ associated with problem $(P).$ 
\end{theorem}
\begin{proof}
It suffices to show (see \cite[Theorem 3.4.6]{Hale},  \cite[Theorem 3.3]{Ball97}) that $\{T(t)\}_{t\geq 0}$ is
 point dissipative (that is, there exists a bounded set $B_{0}$ such that, for any $p_0\in X$,   $T(t)p_0 \in B_{0}$ for all sufficiently large $t$),  and 
 asymptotically compact (that is, for any bounded sequence $p_{0j} \in X$  and  any sequence $t_{j} \to \infty$, the sequence $T(t_{j})p_{0j}$ has a convergent subsequence). That $\{T(t)\}_{t\geq 0}$ is point dissipative follows immediately from the universal upper bound, while the same argument as in the proof of Proposition \ref{p:conveqsubseq} establishes the asymptotic compactness.
\end{proof}

\subsection{Convergence to equilibrium}\label{s:conveq}
In this section we discuss the problem of proving that, for any $p_0\in X$, $T(t)p_0$ converges in $X$ to a unique equilibrium (as opposed to converging to the set of equilibria as established in the previous subsection). This is delicate because there is in general a continuum of equilibria.

For the case when $p_0$ takes finitely many values, as discussed in Section \ref{s:fid}, convergence to a unique equilibrium was proved by  \cite{Pego-92} by a result of  \cite{Hale-Massatt} whose proof was clarified in \cite{Hale-Raugel}. However, adapting the proof to the case of general initial data encounters a serious difficulty which was already noted by \cite{Frie-Mc}, namely that for bounded $\sigma:\mathbb R\to \mathbb R$ (which we effectively have on account of the universal  bounds) the map $p\mapsto \sigma (p)$ is not $C^1$ from $L^2(0,1)$ to $L^2(0,1)$ unless $\sigma$ is constant, this being closely related to the motion of phase boundaries (see also \cite{Brunovsky-Polacik}).

One might think, however, that a possible strategy  might be to use the fact that we have a dense set of initial data for which convergence to a unique equilibrium holds, namely finite-dimensional initial data. However, this kind of argument fails even in finite dimensions, as the following example shows.

\begin{example}
Consider the ODE in ${\mathbb R}^3$ written in cylindrical polars $(r,\theta,z)$ by
\begin{eqnarray*}
\dot r&=&-r(1-r)^2-r|z|\\
\dot z&=& -z|z|\\
\dot\theta&=&r(r-1)
\end{eqnarray*}
Writing $u=(x,y,z)$ we can write this as $\dot u=f(u)$ with $f:{\mathbb R}^3\rightarrow {\mathbb R}^3$ Lipschitz. Also $|u|^2=r^2+z^2$ is a Lyapunov function. So we have global existence and the $\omega$-limit set of every solution is contained in the set of rest points given by $z=0$, $r=0$ or $r=1$, that is by the origin plus the unit circle $S^1$ in the $(x,y)$ plane. Then 
\begin{eqnarray*}
z(t)=\frac{z(0)}{1+|z(0)|t}.
\end{eqnarray*}
Thus if $z(0)\neq 0$ then $z(t) \to 0$ as $t \to \infty$ and since 
\begin{eqnarray*}
\dot r(t)\leq -r(t)|z(t)|,
\end{eqnarray*}
 it follows that 
\begin{eqnarray*}
r(t)\leq \frac{r(0)}{1+|z(0)|t},
\end{eqnarray*}
and hence $u(t) \to 0$ as $t \to \infty$.

If $z(0)=0$ then $z(t)=0$ for all $t\geq 0$ and so we have the ordinary differential equation in ${\mathbb R}^2$
\begin{eqnarray*}
\dot r&=&-r(1-r)^2\\
\dot\theta&=& r(r-1).
\end{eqnarray*}
In this case we have $r(t) \to 1$ as $t \to \infty$ if $r(0)\geq 1$ and $r(t) \to 0$ as $t\to \infty$ if $r(0)<1$. If $r(0)>1$ then $r(t)\leq r(0)$ for all $t\geq 0$ and so $\dot r(t)\geq -r(0)(1-r(t))^2$, from which it follows by integration that
\begin{eqnarray*}
r(t)-1\geq \frac{1}{r(0)t+\frac{1}{r(0)-1}}.
\end{eqnarray*}
Therefore 
\begin{eqnarray*}
\dot \theta \geq\frac{1}{r(0)t+\frac{1}{r(0)-1}},
\end{eqnarray*}
and so $\theta(t)\to \infty$ as $t \to \infty$ and thus $\omega(u(0))=S^1$.

Thus we have an example with a Lyapunov function such that for a dense set of initial data the solution converges to a rest point, while there is a solution which does not tend to a rest point.
\end{example}

Another standard technique for proving convergence to a unique equilibrium is to use the  Lojasiewicz-Simon inequality, introduced by   \cite{Loja} in  a finite-dimensional setting and later generalized by  \cite{Simon} (see also \cite{Jendoubi}) to infinite dimensions, for which analyticity assumptions on nonlinear terms are needed.  The inequality is used to obtain an estimate which in our case would correspond to 
\begin{equation}
\label{l1}
\int_0^\infty \|p_t\|_2\,dt<\infty,
\end{equation}
 thus preventing the length of the orbit being infinite. (Of course,     from   \eqref{finiteint} we have  the weaker statement that $\int_0^\infty \|p_t\|_2^2\,dt<\infty$.) This method does not seem applicable for similar reasons to those mentioned above in connection with the Hale-Massatt theorem. Also, it does not seem easy to prove \eqref{l1} directly.

Because of these difficulties the only currently viable method seems to be that introduced in    \cite{And-Ball} (see also   \cite{Pego-91}). The first step is the following lemma.

\begin{proposition}\label{l:limF}
For any $C^{1}$ function $F : \mathbb{R} \to \mathbb{R}$ we have
\begin{equation*}
\underset{t \to \infty}{\lim}\,\int_{0}^{1} \int_{1}^{p(x,t)}\,F(\sigma(z))\,dz\,dx\,\,\,\,\,\text{exists.}
\end{equation*}
\end{proposition}
\begin{proof}
 Let $F \in C^{1}.$ Then, for $t>0$,
\begin{eqnarray*}
  & & \frac{d}{dt} \int_{0}^{1} \int_{1}^{p(x,t)}\,F(\sigma(z))\,dz\,dx = \int_{0}^{1} F(\sigma(p(x,t)))\,p_{t}(x,t)\,dx \\
  & = & \int_{0}^{1} F(\sigma(p(x,t))\,\left( - \sigma(p(x,t)) + \int_{0}^{1} \sigma(p(y,t))\,dy \right)\,dx \\
  & = & - \int_{0}^{1} \int_{0}^{1} F(\sigma(p(x,t))\,(\sigma(p(x,t)) - \sigma(p(y,t)) )\,dy dx \\
  & = & - \frac{1}{2} \int_{0}^{1} \int_{0}^{1} \big(F(\sigma(p(x,t))) - F(\sigma(p(y,t))) \big) \big(\sigma(p(x,t)) - \sigma(p(y,t)) \big) \,dy dx.
\end{eqnarray*}
If $F'(z) \geq 0,$ then the result immediately follows since from above we get that
\[\frac{d}{dt} \int_{0}^{1} \int_{1}^{p(x,t)}\,F(\sigma(z))\,dz\,dx\,\leq\,0,\] which, by the universal bounds, implies that the function
\[\int_{0}^{1} \int_{1}^{p(x,t)}\,F(\sigma(z))\,dz\,dx\] is nonincreasing and bounded from below. If $F$ is not monotone, then we define
$h(z) = z + \varepsilon\,F(z).$ For sufficiently small $|\, \varepsilon |,$ $r + \varepsilon\,F(r)$ is monotone increasing for $r$ in any compact subset of $(0,\infty)$. Hence $h'(z) \geq 0.$ Since
\begin{eqnarray*}
\varepsilon \int_{0}^{1} \int_{1}^{p(x,t)} F(\sigma(z)) dz dx &=& \int_{0}^{1} \int_{1}^{p(x,t)} h(\sigma(z)) dz dx - \int_{0}^{1} \int_{1}^{p(x,t)} \sigma(z) dz dx \\
&=& \int_{0}^{1} \int_{1}^{p(x,t)} h(\sigma(z)) dz dx -\int_{0}^{1} W(p) dx + C,
\end{eqnarray*}
and each term on the right-hand side tends to a constant as $t\to\infty$ this proves the claim.
\end{proof}
\begin{corollary}\label{charfn}
Let $\chi_{[a,b]}$ be the characteristic function of a bounded closed interval $[a,b]\subset(0,\infty)$. Then
\begin{equation}\label{e:lim-chi}
\underset{t \to \infty}{\lim}\,\int_{0}^{1} \int_{0}^{p(x,t)} \chi_{[a,b]}(\sigma(z))\,dz\,dx\,\,\,\text{exists}.
\end{equation}
 \end{corollary}  
\begin{proof}
It suffices to show that for any sequence $t_j\to\infty$ the sequence $\int_0^1\int_0^{p(x,t_j)}\chi_{[a,b]}(\sigma(z))\,dz$ is Cauchy. Let $F_k$ be a sequence of smooth functions with $F_{k+1}\leq F_k$ and $F_k(s)\to \chi_{[a,b]}(s)$ for all $s$. Then
\begin{eqnarray*}\label{chiproof}
&&\left|\int_0^1\int_0^{p(x,t_j)}\chi_{[a,b]}(\sigma(z))\,dz-\int_0^1\int_0^{p(x,t_l)}\chi_{[a,b]}(\sigma(z))\,dz\right|=\left|\int_0^1\int_{p(x,t_l)}^{p(x,t_j)}\chi_{[a,b]}(\sigma(z))\,dz\right|\\
&&\hspace{.5in}\leq  \left|\int_0^1\int_{p(x,t_l)}^{p(x,t_j)}F_k(\sigma(z))\,dz\right|+\left|\int_0^1\int_{p(x,t_l)}^{p(x,t_j)}[\chi_{[a,b]}(\sigma(z))-F_k(\sigma(z))]\,dz\right|.
\end{eqnarray*}
Since there is a constant $\delta>0$ such that $\delta\leq p(x,t)\leq 1/\delta$ for a.e. $x\in(0,1)$ and all large $t$, the second integral is bounded above by
$$\int_\delta^{\frac{1}{\delta}}[F_k(\sigma(z))-\chi_{[a,b]}(\sigma(z))]\,dz$$ and so, given $\varepsilon>0$, is less than $\varepsilon$ for sufficiently large $k$. But for any such $k$ the first term tends to zero as $j,l\to \infty$ by Proposition \ref{l:limF}, and the result follows.
\end{proof}

\subsubsection{A special cubic case}\label{s:example}

Before dealing with more general cases we give a direct proof of stabilization   when $\sigma = p^{3} - p$, with corresponding $W(p) = \frac{1}{4} (p^{2} - 1)^{2}$,  which as explained in the introduction is of interest in various applications. Of course, this case does not satisfy all of our assumptions (in particular (H2)). However, the proof of Theorem \ref{t:exisuniq} can easily be adapted to get the existence of a semiflow on $X_1=\{q\in L^2(0,1): \int_0^1q\,dx=\mu\}$, and the solution $p=p(x,t)$ satisfies the universal upper bound $|p(x,t)|\leq E(t)$ for all $t>0$.

\begin{proposition}\label{p:ex} 
Let $\sigma(p)=p^3-p$,  $\mu\neq 0$ and  $p_0\in X_1$. Then the unique solution $p=p(x,t)$ to \eqref{peqn} with $p(\cdot,0)=p_0$ satisfies
\[p(x,t) \rightarrow \bar{p}(x)\,\,\,\text{for}\,\,\,\mathrm{a.e.}\,\,x \in (0,1)\,\,\text{as}\,\,\,\,t \to \infty,\] for some equilibrium solution  $\bar {p} \in L^{2}(0,1)$.
\end{proposition}
\begin{proof}
Taking $F(s) = s^{2}$ and using the given form of $\sigma$   we obtain
\begin{eqnarray*}
& & \int_{0}^{1} \int_{1}^{p(x,t)} F(\sigma(z))\,dz\,dx = \int_{0}^{1} \int_{1}^{p(x,t)} (z^{3}- z)^{2}\,dz\,dx \\
& & \qquad \quad =\, \int_{0}^{1} \left(\frac{p^{7}}{7} - \frac{2}{5} p^{5} + \frac{p^{3}}{3} - \frac{8}{105} \right)\,dx.
\end{eqnarray*}
By Lemma \ref{l:limF} we deduce that 
\begin{equation}\label{e:limpoly}
\underset{t \to \infty}{\lim}\,\int_{0}^{1}\left(\frac{p^{7}}{7} - \frac{2}{5} p^{5} + \frac{p^{3}}{3}\right)\,dx =K_1
\end{equation}
for some constant $K_1$. 
  We can rewrite $p^{3}, p^{5}$ and $p^{7}$ in terms of $\sigma(p)$ and $p$ as
\begin{equation}\label{e:trans}
\left\{ \begin{array}{ll}
p^{3} = \sigma(p) + p, & {} \\
p^{5} = \sigma(p)\,p^{2} + \sigma(p) + p, & {} \\
p^{7} = \sigma^{2}(p)\,p + 2\,\sigma(p)\,p^{2} + \sigma(p) + p.
\end{array}\right.
\end{equation}
Substituting into \eqref{e:limpoly} we get
\begin{equation}\label{e:limfromF}
\underset{t \to \infty}{\lim}\int_0^1 \left( \frac{-4}{35}  \sigma(p)\,p^{2}  + \frac{8}{105}\,\mu + \frac{1}{7}   \sigma^{2}(p)\,p  + \frac{8}{105}   \sigma(p)  \right)\,dx=K_1.
\end{equation}
On the other hand, we have
\begin{equation*}
\int_{0}^{1} W(p)\,dx = \int_{0}^{1} \frac{1}{4} (p^{2} - 1)^{2}\,dx  = \frac{1}{4} \int_{0}^{1}\left( \sigma(p)\,p -  p^{2} + 1\right)\,dx.
\end{equation*}
Therefore, by (\ref{energyeq}) (or Lemma \ref{l:limF} with $F(s)=s$)  we deduce that
\begin{equation}\label{e:limenergy}
\underset{t \to \infty}{\lim} \int_{0}^{1}\left( \sigma(p)\,p  -  p^{2}\right)\,dx=K_2
\end{equation}
for some constant $K_2$.
By Proposition \ref{p:conveqsubseq} we know that for a subsequence $t_{j}$ there exists an equilibrium solution $\bar{p}=\bar p(x)$ such that
\[\underset{j \to \infty}{\lim}\,p(x,t_{j}) = \bar{p}(x) \;\;\;\mbox{for a.e. }x\in(0,1).\]
Denoting $\sigma(\bar{p}(x)) = \bar{\sigma}$ and letting $t_{j} \to \infty$ in (\ref{e:limfromF}) and (\ref{e:limenergy}), we obtain using the universal upper bound that
\begin{displaymath}
\left\{\begin{array}{ll}
\displaystyle\frac{-4}{35}\,\bar{\sigma} \displaystyle\int_{0}^{1} \bar {p}^{2}(x)\,dx + \frac{8}{105}\,\mu + \frac{\mu}{7} \,\bar{\sigma}^{2} + \frac{8}{105} \,\bar{\sigma}  = K_{1}, & {} \\
\bar{\sigma}\,\mu - \displaystyle\int_{0}^{1} \bar{p}^{2}(x)\,dx = K_{2}. & {}
\end{array}\right.
\end{displaymath}
  Substituting the second equation into the first  leads to
\begin{equation}\label{e:polysigma}
\mu\bar\sigma^2 - \left(\frac{8}{3} + 4 K_2\right) \bar\sigma - \frac{8}{3} \mu +35 K_1=0.
\end{equation}
This is a second order polynomial in $\bar{\sigma}.$ Hence, since $\mu\neq 0$ and the constants $K_1, K_2$ do not depend on the sequence $t_j$, solving (\ref{e:polysigma}) gives  at most two distinct possible values for $\bar{\sigma}$ in   $\omega(p_0)$.   But there cannot be just two distinct such values since $\omega(p_0)$ is connected. Therefore
\[\sigma(p(x,t)) \rightarrow \bar{\sigma}\;\;\text{as}\,\,\,t \to \infty\]
for a.e. $x\in (0,1)$.
By Corollary \ref{c:bounds}, (H1), and Lebesgue's dominated convergence theorem, this gives 
\begin{equation}\label{e:uniqconv}\nonumber
\int_{0}^{1} \sigma(p(x,t))\,dx\,\rightarrow\,\bar{\sigma}\,\,\,\,\text{as}\,\,\,t \to \infty .
\end{equation}
Thus $p$ satisfies for a.e. $x\in(0,1)$ the ordinary differential equation
$$p_{t}(x,t)\,=\,-\sigma(p(x,t)) + \bar{\sigma} + e(t),$$
where  $e(t) \to 0$ as $t \to \infty$. The result then follows from  
\begin{lemma}{\cite[Lemma 3.4]{Pego-91}}\label{ncpego}
Assume $f: \mathbb{R} \to \mathbb{R}$ is continuous and not constant on any open interval. Assume that $z(t) \in C^{1}(0,\infty)$ is a bounded solution of $z'(t) = f(z(t)) + e(t),$ where $e(t)$ is continuous with $\lim_{t \to \infty} e(t) = 0.$ Then $\lim_{t \to \infty} z(t) = z_{\infty}$ exists, and $f(z_{\infty}) = 0.$
\end{lemma}
\end{proof}

\subsubsection{The nondegeneracy condition}\label{s:genex}
In this subsection we discuss a slightly modified version of the nondegeneracy condition introduced in \cite{And-Ball} and show why it leads to convergence of the solution $p$ to a unique equilibrium as $t\to\infty$. We state this condition as follows:\\

\noindent {\bf Nondegeneracy condition (NC)}

(a) $\sigma:(0,\infty)\to \mathbb R$ is $C^1$ and $\mbox{meas} \mathcal S=0$, where $\mathcal S=\{z:\sigma'(z)=0\}$ is the set of critical points of $\sigma$.

(b)  If $[\alpha,\beta]$ with $\alpha<\beta$ is a closed interval with $[\alpha,\beta]\cap\sigma(\mathcal S)=\emptyset$, and if $p_i, 1\leq i\leq 2k+1, \;k=k(\alpha,\beta)\geq 0$, are the distinct inverse functions to $\sigma$ on $[\alpha,\beta]$, then the derivatives $\{p_i'\}_{1\leq i\leq 2k+1}$ are linearly independent in $C^0([\alpha,\beta])$.

Note that by Sard's theorem $\sigma(\mathcal S)$ is a closed set of measure zero, so that such closed intervals $[\alpha,\beta]$ exist. By (H2), (U1), (U2), on each such interval there are an odd number of distinct inverse functions  and each inverse function is $C^1$.

\begin{proposition}
\label{nondegenthm}
If the nondegeneracy condition holds then each solution $p$ of $(P)$ with $p_0\in X$ converges in $L^2(0,1)$ to an equilibrium $\bar p\in \mathcal E_\mu$ as $t\to\infty$.
\end{proposition}
\begin{proof}
We both simplify and explain more fully the proof in \cite{And-Ball}.
It suffices to show that 
$$c(t)=\int_0^1\sigma(p(x,t))\,dx$$
tends to a limit as $t\to\infty$, since then we can apply Lemma \ref{ncpego}. Assume for contradiction that this is not the case. Since $c(t)$ is bounded for large $t$, there exists an interval $[r,s]$ with $r<s$ such that $c(t)$ takes the values $r$ and $s$ for arbitrarily large values of $t$. Since $c(t)$ is continuous for $t>0$, by Sard's theorem we may assume that $[r,s]\cap\sigma(\mathcal S)=\emptyset$. We then have that the graph of $\sigma$ crosses the interval $[r,s]$ in an odd number $2k+1$ of segments with alternately strictly positive and strictly negative derivatives. If  $k=0$ then we have that for any $c\in[\alpha,\beta]$ there is a sequence $t_j\to\infty$ such that $c(t_j)=c$ for all $j$. By Proposition \ref{p:convtimeder} we have that $\sigma(p(\cdot,t_j))\to c$ in $L^2(0,1)$, and since $\sigma$ is strictly monotone on $[\sigma^{-1}(\alpha),\sigma^{-1}(\beta)]$, it follows that $p(\cdot,t_j)\to \sigma^{-1}(c)$ in $L^2(0,1)$. But then $\int_0^1p(x,t_j)\,dx=\mu=\sigma^{-1}(c)$ for all $c\in[\alpha,\beta]$, a contradiction. Thus we may suppose that $k\geq 1$.  Let $p_i:[r,s]\to\mathbb R,\, 1\leq i\leq 2k+1$ denote the corresponding inverse functions to $\sigma$, which are $C^1$. Thus we have that if $k\geq 1$ then $p_{1}(r)<p_{1}(s)<p_{2}(s)<p_{2}(r)<\cdots <p_{2k+1}(r)<p_{2k+1}(s)$.

 Let $r<\bar r<\bar s<s$ and 
 for $\varepsilon>0$ sufficiently small and $t$ such that $c(t)\in [\bar r,\bar s]$ define 
\begin{equation*}
S_{i}(t) = \{x \in (0,1)\,:\,\,|p(x,t) - p_{i}(c(t))| < \varepsilon\,\},
\end{equation*}
and set $\mu_i(t)=\mbox{meas }S_i(t)$. Then the sets  $S_{i}(t)$ are disjoint, and we claim that 
\begin{equation}\label{e:vol-frac-lim}
\underset{c(t) \in [\bar r ,\bar s]}{\underset{t \to \infty}{\lim}} \sum_{i = 1}^{2k+1} \mu_{i}(t) = 1.
\end{equation}
Indeed,    by the universal bounds there is a $\delta>0$ with $\delta\leq p(x,t)\leq 1/\delta$ for a.e. $x\in (0,1)$ and all large $t$, and there exists $\rho>0$ such that $|\sigma(q)-\gamma|>\rho$ whenever $\gamma\in[\bar r,\bar s]$ and $|q-p_i(\gamma)|\geq\varepsilon$. Hence 
\begin{equation}\label{meas}
\mbox{meas}\,\left((0,1)\setminus \bigcup_{i=1}^{2k+1}S_i(t)\right)\leq\mbox{meas}\,(\{x\in (0,1):|\sigma(p(x,t))-c(t)|>\rho\}).
\end{equation}
But by Proposition \ref{p:convtimeder}, $\sigma(p(\cdot,t))-c(t)\to 0$ in measure as $t\to\infty$, so that the right-hand side of \eqref{meas} tends to zero as $t\to\infty$, proving the claim.

We apply Corollary \ref{charfn} with $a,b$ chosen so that $r<a<\bar r<\bar s<b<s$. Thus 
$$\underset{c(t)\in [\bar r,\bar s]}{\lim_{t\to\infty }}\int_0^1\int_0^{p(x,t)}\chi_{[a,b]}(\sigma(z))\,dz\,dx=\underset{c(t)\in [\bar r,\bar s]}{\lim_{t\to\infty }}\sum_{i=1}^{2k+1}\int_{S_i(t)}\int_0^{p(x,t)}\chi_{[a,b]}(\sigma(z))\,dz\,dx$$
exists, where we have used \eqref{e:vol-frac-lim} and the boundedness of $p$. Thus, by a similar argument to that above,
\begin{equation}\label{limsumchi}\underset{c(t)\in [\bar r,\bar s]}{\lim_{t\to\infty }}\sum_{i=1}^{2k+1}\mu_i(t)\int_0^{p_i(c(t))}\chi_{[a,b]}(\sigma(z))\,dz:=l(a,b)  \mbox{   exists}.
\end{equation}
Now define $$\nu_j(t)=\sum_{l=j}^{2k+1}\mu_l(t), \;\; 1\leq j\leq 2k+1.$$
Then $\mu_{2k+1}(t)=\nu_{2k+1}(t)$ and $\mu_j(t)=\nu_j(t)-\nu_{j+1}(t)$ for $1\leq j\leq 2k$. Thus, by \eqref{limsumchi} the limit as $t\to\infty$ with $c(t)\in[\bar r,\bar s]$ of
\begin{eqnarray*}
&& \sum_{i=1}^{2k+1} \mu_{i}(t)\, \mbox{meas}\,\{[0, p_i(c(t))] \cap \sigma^{-1}([a,b])\} \\
&&\hspace{.5in} = \nu_{2k+1}(t)\, \mbox{meas}\,\{[0, p_{2k+1}(c(t))] \cap \sigma^{-1}([a,b])\} \\
&& \hspace{1in} + \sum_{i=1}^{2k} (\nu_{i}(t) - \nu_{i+1}(t))\, \mbox{meas}\,\{[0,p_{i}(c(t))] \cap \sigma^{-1}([a,b])\}
\end{eqnarray*}
exists and equals $l(a,b)$. After separating the odd and the even terms and using the relation $\sum_{i=1}^{2k+1} \mu_{i}(t) p_{i}(c(t)) = \mu$ we obtain that
\begin{equation}\label{limnu}\underset{c(t)\in [\bar r,\bar s]}{\lim_{t\to\infty }} \sum_{r=1}^{k} \big(\nu_{2r+1}(t) (p_{2r}(a) - p_{2r+1}(a)) + \nu_{2r}(t) (p_{2r-1}(b) - p_{2r}(b))\big)\end{equation}
exists and equals $l(a,b)+p_1(a)-\mu$. We claim that \eqref{limnu} implies  
\begin{equation}\label{limnui}
\underset{c(t)\in [\bar r,\bar s]}{\lim_{t\to\infty }}\nu_i(t):=\bar\nu_i \;\;\;\mbox{   exists for  }1\leq i \leq 2k+1.
\end{equation}
Suppose not.  Then, since the $\nu_i(t)$ are bounded, there are sequences $s_j,t_j\to\infty$ with $s_j, t_j\in[\bar r,\bar s]$ such that the limits $\lim_{j\to\infty}\nu_i(s_j)=\bar\nu_{i}^1$ and $\lim_{j\to\infty}\nu_i(t_j)=\bar\nu_{i}^2$ exist for $1\leq i\leq 2k+1$, but for some $n$ we have  $\bar\nu_{n}^1\neq\bar\nu_{n}^2$. If $n$ is odd, then fixing $b$ and varying $a$ in the interval $[r,\bar r]$ we get
$$\sum_{r=1}^k (\bar\nu_{2r+1}^1-\bar\nu_{2r+1}^2)(p_{2r}(a) - p_{2r+1}(a))=0.$$
Differentiating with respect to $a$ and using (NC)(b) we get that $\bar\nu_{2r+1}^1=\bar\nu_{2r+1}^2$ for $1\leq r\leq k$, a contradiction. We obtain a similar contradiction if $n$ is even, this time fixing $a$ and varying $b$ in the interval $[\bar s, s]$, establishing \eqref{limnui}. Hence $\lim_{t\to\infty, c(t)\in [\bar r,\bar s]}\mu_i(t):=\bar\mu_i$ exists for each $i$, and $\sum_{i=1}^{2k+1}\bar\mu_i=1$ by \eqref{e:vol-frac-lim}.
But 
\begin{eqnarray*}
\underset{c(t)\in [\bar r,\bar s]}{\lim_{t\to\infty }}\int_0^1p(x,t)\,dx&=&\underset{c(t)\in [\bar r,\bar s]}{\lim_{t\to\infty }}\sum_{i=1}^{2k+1}\int_{S_i(t)}p(x,t)\,dx\\
&=&\underset{c(t)\in [\bar r,\bar s]}{\lim_{t\to\infty }}\sum_{i=1}^{2k+1}\mu_i(t)p_i(c(t))\\
&=&\underset{c(t)\in [\bar r,\bar s]}{\lim_{t\to\infty }}\sum_{i=1}^{2k+1}\bar\mu_ip_i(c(t))=\mu.
\end{eqnarray*}
If $c\in[\bar r,\bar s]$ then there is a sequence $t_j\to\infty$ with $c(t_j)=c$ for all $j$. Thus passing to the limit we obtain
$$\sum_{i=1}^{2k+1}\bar\mu_ip_i(c)=\mu \;\;\mbox{for all }c\in[\bar r,\bar s],$$
contradicting (NC)(b) again. This completes the proof.
\end{proof}
\begin{remark}\label{threeroots}
Suppose that $\sigma$ satisfies (NC)(a) and that for any interval $[\alpha,\beta]$ such that $[\alpha,\beta]\cap \sigma(\mathcal S)\neq\emptyset$ there are either one or three roots $p_i(c)$ of $\sigma(p)=c\in [\alpha,\beta]$. Then the proof of Proposition \ref{nondegenthm} shows that $ {\lim_{t\to\infty,\; c(t)\in [\alpha,\beta]}}\mu_i(t)$ exists, without assuming (NC)(b). In fact, supposing, as we may, that $k=1$, we have from \eqref{limnu} that
$$\underset{c(t)\in [\alpha,\beta]}{\lim_{t\to\infty }}\big(\nu_3(t)(p_2(a)-p_3(a))+\nu_2(t)(p_1(b)-p_2(b))\big)\;\;\mbox{exists}.$$
 Suppose for contradiction that $\nu_3(t)$ does not tend to a limit, so that there are sequences $s_j,t_j\to\infty$ with $s_j,t_j\in[\alpha,\beta]$ for all $j$ and such that $\nu_3(s_j)\to \bar\nu_3^1$ and $\nu_3(t_j)\to \bar\nu_3^2$ with $\bar\nu_3^1\neq\bar\nu_3^2$. Then we obtain that $(\bar\nu_3^1-\bar\nu_3^2)(p_2(a)-p_3(a))=0$ for all $a\in [\alpha,\beta]$. But $p_3(a)>p_2(a)$, a contradiction. Similarly $\nu_2(t)$ tends to a limit, and since $\sum_{i=1}^3\mu_i(t)=1$ we deduce that each $\mu_i(t)$ tends to a limit as claimed.
\end{remark}

\begin{figure}[t]
\centerline{
  \includegraphics[scale=0.8]{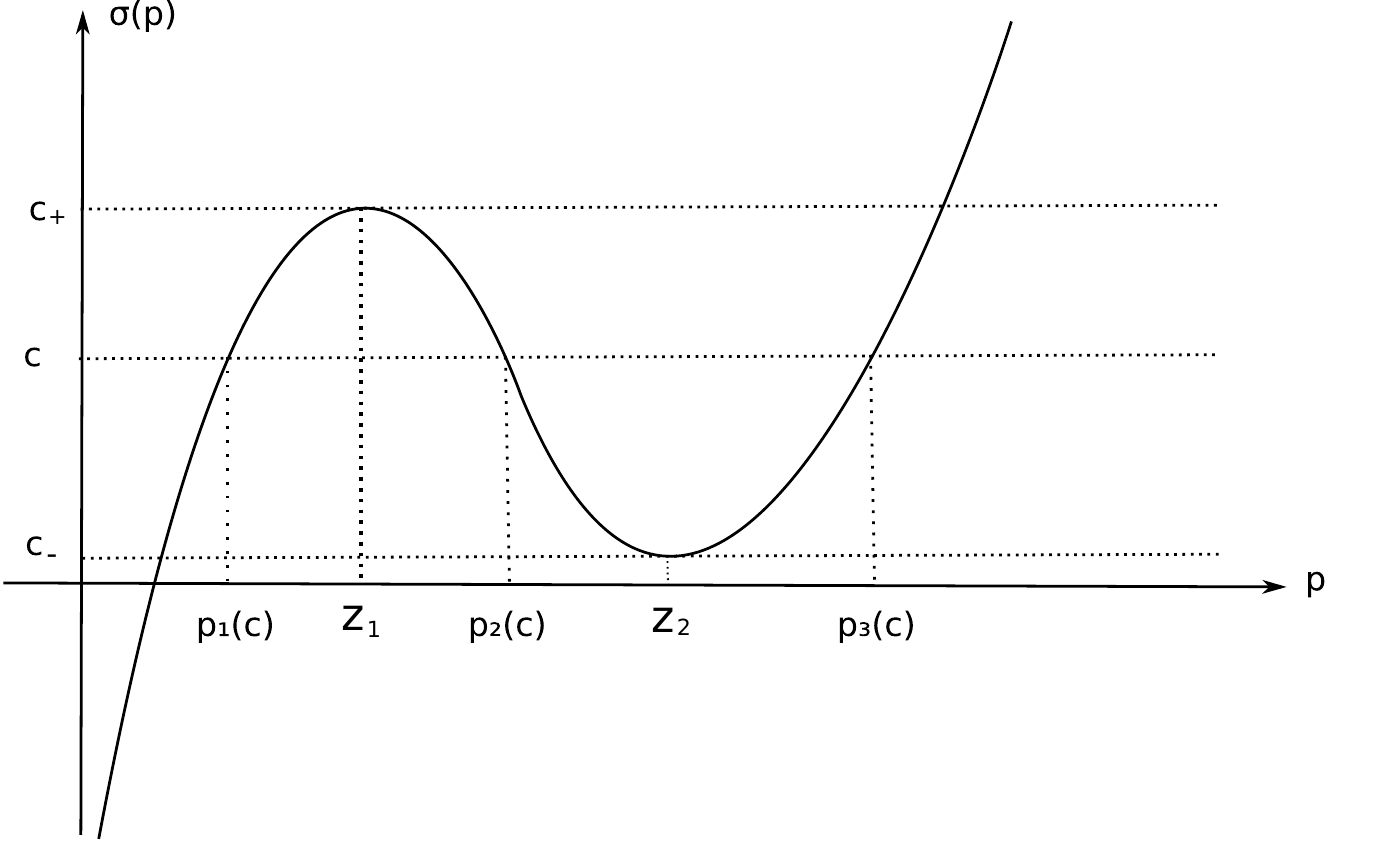}}
 \caption{Setting for the cubic-like stress}\label{stabilization}
\end{figure}

\subsubsection{The case of cubic-like real analytic $\sigma$}
In the case of a cubic-like real analytic $\sigma$ we are able to prove that a weakened form of (NC) holds, that is sufficient for establishing convergence to a unique equilibrium.  In addition to our standing assumptions we make the following hypotheses on $\sigma$:
 \vspace{.05in}

\noindent  (C1)  $\sigma:(0,\infty)\to\mathbb R$ is real analytic.\\
(C2) $\sigma$ has precisely two critical points $z_1<z_2$ with $\sigma(z_2)<\sigma(z_1)$, and $\sigma''(z_1), \sigma''(z_2)$ are nonzero.\vspace{.05in}
 
 Let $c_-=\sigma(z_2), c_+=\sigma(z_1)$. Clearly $\sigma'(z)>0$ for $z\in (0,z_1)$, $\sigma'(z)<0$ for $z\in(z_1,z_2)$ and $\sigma'(z)>0$ for $z\in(z_2,\infty)$.  By the Inverse Function Theorem, for each $c \in (c_{-}, c_{+})$ there exist exactly three roots $p_i(c)$  of $\sigma(p) = c$ with $p_{1}(c) \in (- \infty, z_{1}), p_{2}(c) \in (z_{1}, z_{2})$ and $p_{3}(c) \in (z_{2}, \infty).$ Moreover, each $p_{i}(c)$ is real analytic on $(c_{-}, c_{+})$. (See Fig. \ref{stabilization}.)  
\begin{proposition}\label{cubiclike}
Suppose {\rm (C1), (C2)} hold.  If $\mu_i, \;i=1,2,3$, are nonnegative constants such that
\[\sum_{i = 1}^{3} \mu_{i} p_{i}(c) = \mu, \,\,\, \mbox{for all } c\in[\alpha,\beta]\subset[c_-,c_+],\]
then the $\mu_i$ are all equal.
\end{proposition}
\begin{proof}
Since $\sum_{i = 1}^{3} \mu_{i} p_{i}(c) = \mu, \,\,\, \mbox{for all } c\in[\alpha,\beta]$ and the $p_i(c)$ are real analytic on $(c_{-}, c_{+})$ it follows that 
\begin{equation}\label{convcomb}
 \sum_{i = 1}^{3} \mu_{i} p_{i}(c) = \mu, \,\,\, \mbox{for all } c\in(c_-,c_+).
\end{equation}
Note that since $\sum_{i=1}^3\mu_ip_i'(c)=0$ we have for $c\in(c_-,c_+)$ that 
\begin{eqnarray}\label{e:cond}
  \mu_{1}\frac{p_{1}'(c)}{p_{2}'(c)} + \mu_{2} + \mu_{3}\frac{p_{3}'(c)}{p_{2}'(c)} = 0. 
\end{eqnarray}
As 
$c \to c_{+}$ we have that $p_{1}'(c) \to \infty$, $p_{2}'(c) \to - \infty$, and $p_{3}'(c) \to p_{3}'(c_{+}) = 1/\sigma'(p_{3}(c_{+}))$, where $0<\sigma'(p_3(c_+))<\infty$.  If $\mu_1>0$ we thus have that 
\begin{equation}\label{e:lim-ratio}
\underset{c \to c_{+}}{\lim} \frac{p_{1}'(c)}{p_{2}'(c)} = - \frac{\mu_{2}}{\mu_{1}}.
\end{equation}
By (C2) this gives
\[- \frac{\mu_{2}}{\mu_{1}} = \underset{c \to c_{+}}{\lim} \frac{\sigma'(p_{2}(c))}{\sigma'(p_{1}(c))} = \underset{c \to c_{+}}{\lim} \frac{\sigma''(p_{2}(c)) p_{2}'(c)}{\sigma''(p_{1}(c)) p_{1}'(c)}=\underset{c \to c_{+}}{\lim} \frac{ p_{2}'(c)}{ p_{1}'(c)},\]
since $p_{1}(c)$ and $p_{2}(c)$ both converge to $p_{1}(c_{+}) = p_{2}(c_{+})$ and $\sigma''(p_1(c_+))\neq 0$.
Hence
\[\frac{\mu_{2}^{2}}{\mu_{1}^{2}} = 1.\]
This means either $\mu_{1} = - \mu_{2},$ which is impossible since $\mu_{1},\mu_{2} > 0,$ or $\mu_{1} = \mu_{2}$ as required. If, on the other hand, $\mu_1=0$, then $\mu_2=0$ by \eqref{e:cond}, which, since $p_3(c)$ is not constant, implies $\mu_3=0$, in contradiction to $\mu>0$.  Arguing similarly for $c \to c_{-}$ we  also obtain $\mu_{2} = \mu_{3}.$
\end{proof}
The weakened form of (NC) can now be stated as follows:\vspace{.05in}

\noindent{\bf(NC3)}   For some $c\in(c_-,c_+)$ there holds
$$\frac{1}{3}\sum_{i=1}^3p_i(c)\neq \mu.$$
\begin{theorem}\label{cubicconv}
Assume {\rm (C1), (C2)} and {\rm (NC3)} hold. Then for any $p_0\in X$ we have $T(t)p_0\to\bar p$ as $t\to\infty$ for some $\bar p\in \mathcal E_\mu$.
\end{theorem}\begin{proof}
Suppose not. By Lemma \ref{ncpego} it is enough to show that $c(t)$ converges as $t\to\infty$. If not, then there exists an interval $[\alpha,\beta]$ with $\alpha<\beta$ and $c(s_j)=\alpha, c(t_j)=\beta$ for sequences $s_j,t_j\to\infty$. By the same argument as in Proposition \ref{nondegenthm} we may assume that $[\alpha,\beta]\subset(c_-,c_+)$. Then by Remark \ref{threeroots} we have that $\lim_{t\to\infty,\; c(t)\in [\alpha,\beta]}\mu_i(t):=\mu_i$ exists for $i=1,2,3$. Since 
$$\underset{c(t)\in [\alpha,\beta]}{\lim_{t\to\infty}}\sum_{i=1}^{3}\mu_i(t)p_i(c(t))=\mu$$
 we have that $\sum_{i=1}^3\mu_i p_i(c)=\mu$ for all $c\in[\alpha,\beta]$. Thus by Proposition \ref{cubiclike} we have that the $\mu_i$ are all equal, and since $\sum_{i=1}^3\mu_i=1$ we have that $\mu_1=\mu_2=\mu_3=\frac{1}{3}$. Hence $\frac{1}{3}\sum_{i=1}^3p_i(c)=\mu$ for all $c\in [\alpha,\beta]$, and thus by real analyticity for all $c\in(c_-,c_+)$. This contradicts (NC3).
\end{proof}
\begin{remark}
We could also have used this method to prove Proposition \ref{p:ex}. Indeed (NC3) holds because $\sum_{i=1}^3p_i(c)=0$ and $\mu\neq 0$. 
\end{remark}

\section{Discussion}
\label{full}

  In this work we analyzed the quasistatic problem corresponding to the equation \eqref{e:visco} with $S(y_{x}, y_{xt}) = y_{xt},$ which allowed us to make a connection with the theory of gradient flows. It would be interesting to generalize our analysis to the case of quasistatic motion for more general $S$ that are not linear in $y_{xt}$. In order to comment on possible extensions of our results to the fully dynamical problem, we need to consider two separate issues, namely the existence and the asymptotic behaviour of solutions. We discuss this first for the one-dimensional case with $S(y_{x}, y_{xt}) = y_{xt}$.\vspace{.05in}

\noindent (i) Existence: The global existence of solutions for the equation expressed in terms of the displacement $u(x,t)=y(x,t)-\mu x$, that is
\begin{equation}\label{e:full-dyn-S}
u_{tt} = \sigma(u_{x})_{x} + u_{xxt}, \quad x \in (0,1),\; t > 0,
\end{equation}
with boundary conditions
\begin{equation}\label{dirichlet}
u(0,t)=u(1,t)=0,
\end{equation}
was analyzed by   \cite{Pego} using a nonlocal  transformation inspired by \cite{And}. For the case of displacement boundary conditions (see \cite{BHJPS}) the transformation is given by 
\[w(x,t) = \int_{0}^{x} u_{t}(s,t)\, ds-\int_0^1\int_0^xu_t(s,t)\,ds\,dx, \quad q(x,t) = u_{x}(x,t) - w(x,t).\]
Then $w$ and $q$ form a solution to the problem
\begin{eqnarray}w_{t} &=& w_{xx} + \mathcal F(p + w),\label{eqnmotion1}\\ \qquad q_{t} &=& - \mathcal F(q + w),\label{eqnmotion2}\end{eqnarray}
where $\mathcal F(z)=\sigma (z)-\int_0^1\sigma(z)\,dx$.  Thus \eqref{e:qseqn} corresponds to formally setting $p=q+w$ in \eqref{eqnmotion2} and neglecting $w_t$, which is an integrated form of the inertia. This transformation gives a global existence theory for \eqref{e:full-dyn-S} with initial data $u(\cdot,0)\in W_0^\infty(0,1), u_t(\cdot,0)\in L^2(0,1)$ when $\sigma=\sigma(z)$ is defined for all $z \in \mathbb R$ and satisfies suitable conditions as $|z|\to\infty$. A number of authors have proved the global existence and uniqueness of solutions to systems of nonlinear viscoelasticity   satisfying $\mbox{ess inf}_{x\in(0,1)} y_x(x,t)  >0$ for $t>0$, provided $\mbox{ess inf}_{x\in(0,1)} y_x(x,0)  >0$ (see, for example, \cite{Andthesis}, \cite{Ant-Seid},  \cite{Ant-Seid1}, \cite{Daf}, \cite{Watson}). However, there does not seem to be any version of our universal lower (or upper) bound, and this would be interesting to investigate.  \vspace{.05in}

\noindent (ii) Stability: The asymptotic behaviour of solutions to \eqref{e:full-dyn-S}, \eqref{dirichlet} as $t\to\infty$ was studied in particular by   \cite{And-Ball}, who showed that under the nondegeneracy assumption (NC) $u_x(\cdot,t)$ converges in the sense of Young measures to a unique Young measure $(\nu_x)_{x\in(0,1)}$. In fact, although it does not seem to have been explicitly noted in the literature, under (NC) and other hypotheses, every solution is such that $u_x(\cdot,t)$ converges boundedly almost everywhere to a unique equilibrium as $t\to \infty$. This follows by noting that (NC) implies that $c(t)=\int_0^1\sigma(u_x(x,t))\,dx$ tends to a limit, and that this implies that $u_x(x,t)$ does so for each $x$ on account of \eqref{eqnmotion2} and an easy modification of Lemma \ref{ncpego} in which $f(z(t))$ is replaced by $f(z(t)+q(t))$ with $\lim_{t\to\infty}q(t)=0$. Note that this argument gives relative compactness of positive orbits via (NC), whereas in our problem we can prove relative compactness without this using Helly's theorem.

For the equations of nonlinear viscoelasticity of rate type in three space dimensions there is currently no global existence theory for solutions for frame-indifferent constitutive equations, except that of \cite{Po-Fe-82} for solutions of small energy. In the case of the isothermal stress constitutive law $S(Dy,Dy_t)=D_AW(Dy)+\mu Dy_t$, where $W=W(A)$ is the elastic stored-energy function and $\mu>0$, which is not frame-indifferent, there is an existence theory due to \cite{Rybka92} (see also \cite{Tvedt} for a theory allowing nonlinear dependence on the velocity gradient $Dy_t$, but which contravenes frame-indifference). As regards the asymptotic behaviour of solutions in situations corresponding to solid phase transformations almost nothing is known. The key issue  is whether solutions $y=y(x,t)$ typically generate (local or global) minimizing sequences $y(\cdot,t_j)$ for the energy for sequences $t_j\to\infty$. This was shown never to be the case for the modification of \eqref{e:full-dyn-S} in which the term $-\alpha u$ is added to the right-hand side, where $\alpha>0$. Also  \cite{Frie-Mc-97} show that for \eqref{e:full-dyn-S}, \eqref{dirichlet} there are dynamically stable equilibria which are not local minimizers of the energy, so that for initial data  close to such an equilibrium the solution does not generate a minimizing sequence, a result that probably extends to \eqref{bc2}.
The numerical calculations of \cite{Swart-Holmes} suggest that dynamical generation of minimizing sequences is more likely in higher dimensions. There seem to be no general techniques for deciding whether positive orbits in such problems are relatively compact or not. If all such orbits are relatively compact then we can expect every solution to have  an $\omega$-limit set consisting just of equilibrium solutions. In this connection we mention the recent result of \cite{Norton}, who shows the existence of infinitely many equilibrium solutions for a model 2D problem.

\section*{Acknowledgements}
We are grateful to Gero Friesecke, Bob Pego and Endre S\"uli  for useful discussions. We also thank the referee for valuable comments.
The research of both authors was partly supported by EPSRC grant EP/D048400/1 and by the
EPSRC Science and Innovation award to the Oxford Centre for Nonlinear PDE (EP/E035027/1).
The research of JMB was also supported by the European Research Council under the European
Union's Seventh Framework Programme (FP7/2007-2013) / ERC grant agreement no 291053 and
by a Royal Society Wolfson Research Merit Award. The research of Y\c{S} was also supported by T\"UB\.ITAK fellowship 2213.

\bibliographystyle{apalike} 
\bibliography{bibliography}

\end{document}